\documentclass{amsart}
\usepackage{amssymb,amsmath,amsthm,epsf,epsfig,dsfont,bbm}

\usepackage{latexsym}
\usepackage{upref, eucal}
\usepackage[all]{xy}

\newcommand {\nc} {\newcommand}
\newcommand {\enm} {\ensuremath}

\def \d{\delta}

\nc {\bdm} {\begin{displaymath}}
\nc {\edm} {\end{displaymath}}
\newtheorem {theorem} {\bf{Theorem}}[section]
\newtheorem {lemma}[theorem] {\bf Lemma}
\newtheorem {proposition}[theorem] {\bf Proposition}

\newtheorem {corollary}[theorem] {\bf Corollary}
\numberwithin {equation}{section}

     \newcommand\fp{\mathfrak{p}}

\newcommand{\Ou}{\enm{\mathcal{O}}}

\nc{\J}{\enm{\mathcal{J} }}
\nc {\Z} {\enm{\mathbb{Z}}}
\nc {\form}[1] {\enm{\mbox{\underline{for}}}_{#1}}
\nc {\prol}[1] {\enm{\mbox{\underline{prol}}_{{#1}^*}}}

\nc {\stk} {\stackrel}

\newcommand{\map}{\rightarrow}

\newcommand{\inj}{\hookrightarrow}
\newcommand{\Pn}[2] {\ensuremath{ {\mathbb{P}}^{#1}_{#2}}}
\nc{\Quot}[3]{\enm{ {\mathfrak{Quot}_{ {#1}/{#2}/{#3}}}}}
\nc{\Hilb}[2]{\enm{ {\mathfrak{Hilb}_{ {#1}/{#2}}}}}
\newcommand{\mfrak}[1]{\mathfrak{#1}}

\newcommand{\bb}[1]{\mathbb{#1}}
\newcommand{\mcal}[1]{\mathcal{#1}}

\nc {\Coh}[4] {\ensuremath{H^{#1}(\Pn{#2}{},{#3}({#4}))}}
\nc {\Ch}[3] {\enm{H^{#1}(X_t,{#2}_t({#3}))}}
\nc {\Qphi}[4]{\enm{ {\mathfrak{Quot}^{~#4}_{ {#1}/{#2}/{#3}}}}}
\nc {\Gra}[4]{\enm{ {\mathfrak{Grass}_{#2}({#3},{#4})}}}
\nc {\HomA}[2]{\enm{\mathrm{Hom}_A{#1}{#2}}}
\nc {\tr}{\mathrm{tr}}

\newcommand{\eg}{e}

\nc {\C}[2]{\enm{\left(\begin{array}{l} {#1} \\ {#2} \end{array} \right)}}
\nc {\mat}[4]{\enm{\left(\begin{array}{ll}{#1} & {#2} \\ {#3} & {#4}
\end{array}\right)}}

\def \vp{\varphi}
\def \mb{\mbox}

 \def \Z{{\mathbb Z}}

   \def \h{\hat{\ }}

\def \d{\delta} \def \bZ{{\mathbb Z}}  \def \bF{{\bf F}}

  \def \bX{{\bf X}} \def \bH{{\bf H}}
 \def \bF{{\mathbb F}}

\def \hG{\hat{\mathbb{G}}_{\mathrm{a}}} 
\def \Ga{\mathbb{G}_{\mathrm{a}}}
\def \hA{\hat{A}}

\def \R1{R((q))[q']\h}

\newcommand{\xg}{h}

\usepackage{xcolor}

\newcommand{\lam}{\lambda}
\DeclareMathOperator{\Spec}{\mathrm{Spec}}
\DeclareMathOperator{\Spf}{\mathrm{Spf}}
\DeclareMathOperator{\Lie}{\mathrm{Lie}}

\newcommand{\Hom}{\mathrm{Hom}}
\newcommand{\End}{\mathrm{End}}
\newcommand{\Ext}{\mathrm{Ext}}
\newcommand{\forl}{\mathrm{for}}
\newcommand{\bI}{{\bf I}}
\newcommand{\switt}{s_{\mathrm{Witt}}}
\newcommand{\teich}{v}
\newcommand{\sO}{\mathcal{O}}
\newcommand{\longlabelmap}[1]{{\,\buildrel #1\over\longrightarrow\,}}
\newcommand{\longmap}{{\,\longrightarrow\,}}
\def\longisomap{{\,\buildrel \sim\over\longrightarrow\,}} 
\def\isomap{{\,\buildrel \sim\over\rightarrow\,}} 
\newcommand{\ret}{\rho}
\newcommand{\oldmarginpar}[1]{}

\newcommand{\xqa}{{q}}
\newcommand{\xqb}{{\hat{q}}}

\newcommand{\rdeg}{f}

\title[Differential characters of Drinfeld Modules]{Differential characters of Drinfeld Modules and de~Rham cohomology}
\author{James Borger \and Arnab Saha}
\date{}
\email{james.borger@anu.edu.au, arnabsaha0930@gmail.com}
\address{Australian National University, Max Planck Institute for Mathematics}

\begin{document}
\maketitle

\begin{abstract}
We introduce differential characters of Drinfeld modules. These are function-field analogues of Buium's
$p$-adic differential characters of elliptic curves and of Manin's differential characters of
elliptic curves in differential algebra, both of which have had notable Diophantine applications. We
determine the structure of the group of differential characters. This shows the existence of a family of
interesting differential modular functions on the moduli of Drinfeld modules. It also leads to a canonical
$F$-crystal equipped with a map to the de Rham cohomology of the Drinfeld module.
This $F$-crystal is of a differential-algebraic nature, and the relation to the classical cohomological
realizations is presently not clear.
\end{abstract}

\section{Introduction}

The theory of arithmetic jet spaces developed by Buium draws 
inspiration from the theory of differential algebra over a function field.
In differential algebra, given a scheme
$E$ defined over a function field $K$ with a derivation $\partial$ on it, one can define the jet spaces $J^nE$ for all $n \in
\bb{N}$ with respect to $(K,\partial)$ and they form an inverse system of 
schemes satisfying a universal property with respect to derivations lifting 
$\partial$. The ring of global functions
$\Ou(J^nE)$ can be thought of as the ring of $n$-th order differential functions
on $E$. In the case when $E$ is an elliptic curve and its structure sheaf
$\Ou_E$ does not have a derivation lifting $\partial$ (if it does, then it is
the isotrivial case and $E$ will descend to the subfield $K^{\partial=0}$ of constants), there exists a 
differential function $\Theta \in \Ou(J^2E)$ which is a homomorphism of group schemes from $J^2E$ to
the additive group $\bb{G}_a$.  Such a $\Theta$ is an example of a differential
character of order $2$ for $E$ and is known as a Manin character.
Explicitly, if $E$ is given by the 
Legendre equation $y^2=x(x-1)(x-t)$ over $K=\bb{C}(t)$ with derivation 
$\partial = \frac{d}{dt}$,  then
	$$ 
	\Theta(x,y,x',y',x'',y'')= \frac{y}{2(x-t)^2} - \frac{d}{dt}
	\left[2t(t-1)\frac{x'}{y} \right] + 2t(t-1)x'\frac{y'}{y^2}. 
	$$
The existence of such a $\Theta$ is a consequence of the Picard--Fuchs equation.
Using the derivation $\partial$ on $K$, we can lift any $K$-rational point $P \in E(K)$ 
canonically to $J^2E(K)$,
and this defines a homomorphism $\nabla:E(K) \map J^2E(K)$. We emphasize that
$\nabla$ is merely a map on $K$-rational points and does not come from a map of schemes. The 
composition $\Theta \circ \nabla: E(K) \map \bb{G}_a(K)$ is then a group homomorphism
of $K$-points. 
Note that the torsion points of $E(K)$ are contained in the kernel of $\Theta$
since $\bb{G}_a(K)$ is torsion free. Such a $\Theta$ was used by Manin
to give a proof of the Lang--Mordell conjecture for abelian varieties over
function fields \cite{M}. Later Buium gave a different proof, using other
methods, but still using the Manin map \cite{Bui1}.

The theory of arithmetic jet spaces, as developed by Buium, proceeds similarly. Derivations $\partial$ are
replaced by what are known as $\pi$-derivations $\d$. They naturally arise from the theory of $\pi$-typical
Witt vectors. For instance, when our base ring $R$ is an unramified extension of the ring of $p$-adic integers
$\Z_p$, for a fixed prime $\pi= p$, the Fermat quotient operator $\d x = \frac{\phi(x)-x^p}{p}$ is the unique
$p$-derivation, where the endomorphism $\phi\colon R\to R$ is the lift of the $p$-th power Frobenius
endomorphism of $R/pR$. In analogy with differential algebra, one can define the $n$-th order jet space $J^nE$
of an elliptic curve $E$ over $R$ to be the ($\pi$-adic) formal scheme over $R$ with functor of points
$$(J^nE)(C) = \Hom_R(\Spec W_n(C),E), $$ where $W_n(C)$ is the ring of $\pi$-typical Witt vectors of length
$n+1$, which we view as the arithmetic analogue of $C[t]/(t^{n+1})$. The jet space $J^nE$ is also known as the
Greenberg transform. As with the differential jet space, it has relative dimension $n+1$ over the base, in this
case $\Spf R$.

Then one can define $\bX_n(E)$ to be the $R$-module of all group-scheme homomorphisms from $J^nE$ to the
$\pi$-adic formal scheme $\hG$. Let $\bX_\infty(E)$ be the direct limit of the $\bX_n(E)$. Now the usual
Frobenius operator on Witt vectors induces a canonical Frobenius morphism $\phi:J^{n+1}E \map J^nE$ lying over
the endomorphism $\phi$ of $\Spf R$. Hence pulling back morphisms via $\phi$ as $\Theta \mapsto \phi^*\Theta$,
endows $\bX_\infty(E)$ with an action of $\phi^*$ and hence makes $\bX_\infty(E)$ into a left module over the
twisted polynomial ring $R\{\phi^*\}$ with commutation law $\phi^*\cdot r=\phi(r)\cdot \phi^*$. In \cite{Bui2},
Buium studied the structure of $\bX_\infty(E)$. Putting
$K=R[\frac{1}{p}]$, he showed that $\bX_\infty(E) \otimes_R K$ is freely generated by a single element as a
$K\{\phi^*\}$-module. This element is of order $2$ unless $E$ has a Frobenius lift (in particular
is a canonical lift of an ordinary curve), in which case it is of order
$1$. It is the arithmetic analogue of the Manin character.

In this paper, we study the function-field analogue of Buium's theory. We emphasize that we take the
function-field analogue in every possible sense. So instead of looking at characters $J^nE\to\hG$ of
$\bZ$-module schemes over $\bZ_p$, where the $\bZ$-module scheme $E$ is an elliptic curve over $\bZ_p$ and
$J^nE$ is its $p$-typical arithmetic jet space defined above, we will look at, for example, characters
$J^nE\to\hG$ of ($t$-adically formal) $\bF_q[t]$-module schemes over $\bF_q[[t]]$, where $E$ is a Drinfeld
$\bF_q[t]$-module, $\hG$ is the additive group with the tautological $\bF_q[t]$-module structure, and $J^nE$ is
its function-field arithmetic jet space---in other words, the Greenberg transform but with ``$t$-typical'' Witt
vectors. The most important result in this paper is the construction of a canonical $F$-crystal $\bH(E)$ which
comes with a Hodge-type filtration and a morphism $\bH(E) \map \bH_{\mathrm{dR}}(E)$ to the usual de Rham
cohomology preserving the filtration. As a consequence of the methods that go into the construction of
$\bH(E)$, we also prove that $\bX_\infty(E)$ is freely generated by a single element as an $R\{\phi^*\}$-module,
which is a stronger, integral version of the equal-characteristic analogue of Buium's result. Here, we would
like to emphasize that all the fundamental principles that go into our approach also work for $p$-adic elliptic
curves.

%



\vspace{3mm}

Before we describe our main results in detail, we wish to fix a few notations. Let $\bF_q$ be the finite field
with $q$ elements and $A$ is the coordinate ring of $X \backslash\{\infty\}$, where $X$ is a projective,
geometrically connected, smooth curve over $\bF_q$ and $\infty$ a $\bF_q$-point on it. Let $\mfrak{p}$ be a
fixed maximal ideal of $A$, and let $\pi$ be an element of $\mfrak{p}\setminus\mfrak{p}^2$. Let $R$ be an
$A$-algebra which is a complete discrete valuation ring with maximal ideal $\pi R$ and which has a lift
$\phi:R\to R$ of the $\xqb$-power Frobenius from $R/\pi R$, where $\xqb=|A/\mfrak{p}|$. Then one can consider
the operator on $R$ given by $\d x = \frac{\phi(x) - x^{\xqb}}{\pi}$. It is called the \emph{$\pi$-derivation}
associated to $\phi$.

Then as in the mixed-characteristic case above, one can define the $t$-typical Witt vectors and hence the
$t$-typical arithmetic jet space functor. For any (formal) $A$-module scheme $E$ over $R$, the jet space also
$J^nE$ has a natural (formal) $A$-module-scheme structure. However, we would like to remark here that for all
$n \geq 1$, the $J^nE$ are not abelian Anderson $A$-modules (as defined in~\cite{Hartl}, 1.2). Then we 
let $\bX_n(E)$ denote the set of $A$-linear
differential characters of order $n$, that is, the set of homomorphisms $J^nE\to \hG$ of (formal) $A$-module
schemes over $R$. Finally, we form their direct limit $\bX_\infty(E)$, which is naturally an
$R\{\phi^*\}$-module, as above.

We say $E$ splits at $m$ if $\bX_m(E) \ne \{0\}$ but $\bX_i(E) = \{0\}$ for all $0\leq i \leq m-1$. Then we
show that $m$ satisfies $1 \leq m \leq r$, where $r$ is the rank of $E$, and that $\bX_m(E)$ is a free
$R$-module with a canonical basis element $\Theta_m \in \bX_m(E)$, depending only on our chosen coordinate on
$E$. In the case when the rank $r$ is $2$, we have $m=2$ unless $E$ admits a lift of Frobenius compatible with
the $A$-module structure on $E$, in which case $m=1$. Then our first main theorem is a strengthened version of
the equal-characteristic analogue of Buium's result in \cite{Bui2}.

\begin {theorem}
\label{phigen-intro}
Let $E$ be a Drinfeld module that splits at $m$. Then the $R$-module $\bX_m(E)$ is free of 
rank $1$, and it freely generates $\bX_\infty(E)$ as an $R\{\phi^*\}$-module
in the sense that the canonical map $R\{\phi^*\}\otimes_R \bX_m(E) \map 
\bX_\infty(E)$ is an isomorphism.
\end{theorem}

Let us now proceed to our second result.
Let $u:J^nE \map E$ be the usual projection map and put $N^n =\ker{u}$. Since $u$ is $A$-linear, $N^n$ is a
formal $A$-module scheme of relative dimension $n$ over $\Spf{R}$. For each $n \geq 1$, we show in proposition
\ref{latfrob} that there is a lift of Frobenius $\mfrak{f}:N^{n+1} \map N^n$ making the system $\{N^n\}$ into a
prolongation sequence with respect the obvious projection map $u:N^{n+1} \map N^n$. We call $\mfrak{f}$ the
{\it lateral Frobenius}. However, $\mfrak{f}$ is not compatible with $i$ and $\phi:J^{n+1}E \map J^nE$ in the
obvious way, that is, it is not true that $\phi \circ i = i \circ \mfrak{f}$ holds. In fact, we can not expect
it to be true because that would induce an $A$-linear lift of Frobenius on $E$ which is not the case to start
with. Instead we have
	$$
	\phi^2 \circ i = \phi \circ i \circ \mfrak{f}.
	$$ 

In section~\ref{sec-deRham}, we construct a canonical $F$-crystal attached to $E$. The
$F$-crystal, denoted $\bH(E)$, is an $R$-module which has a semi-linear operator $\mfrak{f}^*$ (induced from
$\mfrak{f}$) on it and is of rank $m$, which we emphasize can be strictly smaller than $r$.
(By the term $F$-crystal, we mean only a free $R$-module of finite rank equipped with a semi-linear
operator $F$. We do not assume $F$ is injective, although on $\bH(E)$ this will be true
generically. The reader can refer to~\cite{Laumon-book-vol1}, section (2.4).) The module $\bH(E)$
also has a Hodge-type filtration and canonically maps to the de Rham cohomology of $E$, with its
Hodge filtration.

\begin{theorem}
\label{fullcrys-intro}
There is a canonical map between exact sequences
$$
\xymatrix{
0 \ar[r] & \bX_m(E) \ar[r] \ar[d]_-\Upsilon & 
\bH(E) \ar[r] \ar[d]^-{\Phi} & \bI(E) \ar[r] \ar@{^{(}->}[d]& 0\\
0 \ar[r] &\Lie(E)^* \ar[r] & \bH_{\mathrm{dR}} (E) \ar[r] & \Ext(E,\hG)
 \ar[r] & 0
}
$$
Moreover, the operator $\mfrak{f}^*$ on $\bH(E)$ descends to its image under
$\Phi$.
\end{theorem}

\noindent The definitions of the maps $\Upsilon$ and $\Phi$
are given in (\ref{diag-crys}), and the proof
is given in section~\ref{subsec-ourdeRham}.
There is a close connection between these two theorems---in fact, 
our proof of theorem~\ref{phigen-intro} goes
by way of theorem~\ref{fullcrys-intro}. 

Finally, we conclude the paper with some explicit computations of the
structure constants of the $F$-crystal $\bH(E)$, which are new differential modular forms.

\vspace{3mm}
To a Drinfeld module $E$, the crystalline theory also attaches an $F$-crystal $\bH_{\mathrm{crys}}(E)$. It
appears that our $\bH(E)$ has subtle connections with $\bH_{\mathrm{crys}}(E)$, but it also appears that any
such connection would be indirect. This is because $\bH(E)$, unlike $\bH_{\mathrm{crys}}(E)$, has a
fundamentally differential-algebraic nature in that it lies not over a point of the moduli space of Drinfeld
modules but over a point of the jet space of the moduli space. For instance, the computations in
section~\ref{sec-computation2} show the structure constants of $\bH(E)$ do involve the higher $\pi$-derivatives
of the structure constants of the Drinfeld module.
The phenomenon of $\pi$-differential invariants depending on higher $\pi$-derivatives of modular parameters in 
the mixed-characteristic setting can be found in \cite{bosa2},\cite{Bui2},\cite{BuSa1}, 
\cite{BuSa2}, \cite{BuSa3}, \cite{BuSa4}.

It would be interesting to understand the exact nature of
the relationship between $\bH(E)$ and the crystalline cohomology groups, as well as the \'etale cohomology
groups and the other constructions in $\pi$-adic Hodge theory. This is all the more true because, as we remarked
before, the techniques developed in this paper have analogues for $p$-adic elliptic curves~\cite{bosa2},
and as a
result, we do obtain an analogous construction of the $F$-crystal $\bH(E)$ for elliptic curves.

{\bf Acknowledgement.} We wish to thank the anonymous referee for carefully 
reading our 
article and the suggestions which led to deeper clarifications and brought 
more lucidity in our present version of the paper.

\section{Notation}

Let us fix some notation which will hold throughout the paper. Let $\xqa=p^h$ where $p$ is a prime and 
$h \geq 1$. Let $X$ be a projective, geometrically connected, smooth curve over $\bF_\xqa$. Fix an 
$\bF_\xqa$-rational point $\infty$ on $X$. Let $A$ denote the Dedekind domain $\Ou(X \setminus \{\infty\})$. 
Let $\mfrak{p}$ be a maximal ideal of $A$, and let $\hA$ denote the $\mfrak{p}$-adic completion of $A$. 
Let $t$ be an element of $\mfrak{p}\setminus\mfrak{p}^2$,
and let $\pi$ denote its image in $\hA$. Then $\pi$ generates the maximal ideal $\hat{\mfrak{p}}$ of $\hA$.
Let $k$ denote the residue field $A/\fp$, and
let $\xqb$ denote its cardinality.
So, for example, if $A=\bF_q[u]$ and $\fp=(t)$, where $t\in\bF_q[u]$ is an irreducible polynomial,
then $\xqb=q^{\deg(t)}$. Note that the quotient map $\hA\to k$
has a unique section. Thus $\hA$ is not just an $\bF_\xqa$-algebra but also canonically a $k$-algebra.

Now let $R$ be an $\hA$-algebra which is $\fp$-adically complete and flat, or equivalently $\pi$-torsion free.
Thus the composition 
	\begin{equation}
	\label{eq:theta3849}
	\theta: A \map \hA\map R
	\end{equation}
is injective (assuming $R\neq \{0\}$)
and hence one says that $\theta$ is of {\it generic
characteristic}. Let us also fix an $\hA$-algebra endomorphism $\phi:R\to R$ which lifts the $\xqb$-power
Frobenius modulo $\mfrak{p}R$:
	$$
	\phi(x)\equiv x^\xqb\bmod \mfrak{p}R.
	$$
Do note that the identity map on $\hA$ does indeed lift the $\xqb$-power Frobenius on $\hA/\hat{\mfrak{p}}$. 	

For our main results, $R$ will in the end be a discrete valuation ring, most importantly the completion
$\bar{\bF}_q[[\pi]]$ of the maximal unramified extension of $\hA$, where $\phi$ satisfies $\phi(c)=c^\xqb$ for
$c\in\bar{\bF}_q$ and $\phi(\pi)=\pi$. So the reader may assume this from the start. (Also note that not all
rings $R$ admit such a Frobenius lift; so the existence of $\phi$ does place a restriction on $R$.) But some
form of our results should hold in general, and with essentially the same proofs. This is of some interest, for
instance when $R$ is the coordinate ring of the orindary locus of the moduli space of Drinfeld modules of a
given rank. (For the representability of Drinfeld modular varieties, see Laumon's book~\cite{Laumon-book-vol1},
theorem 1.4.1.) With an eye to the future, we have not assumed that $R$ is a discrete valuation ring where it
is easily avoided, in sections \ref{sec-witt-vectors}--\ref{sec-lateral}.

Let $K$ denote $R[1/\pi]$, and for any $R$-module $M$ write $M_K=K\otimes_R M$.
Finally, let $S$ denote $\Spf R$.

\section{Function-field Witt Vectors}
\label{sec-witt-vectors}
Witt vectors over Dedekind domains with finite residue fields were introduced in \cite{bor11a}. We
will only work over $\hA$, which is the ring of integers of a local field of characteristic $p$, and here they
were introduced earlier in~\cite{drin76}. The basic results can be developed exactly as in any of the usual
developments of the $p$-typical Witt vectors. The only difference is that in all formulas any $p$ in a 
coefficient is replaced with a $\pi$ and any $p$ in an exponent is replaced with a $\xqb$. 

\subsection{Frobenius lifts and $\pi$-derivations}
Let $B$ be an $R$-algebra, and let $C$ be a $B$-algebra with structure map $u:B \map C$. 
In this paper, a ring homomorphism $\psi:B\map C$ will be called a 
{\it lift of Frobenius} (relative to $u$) if it satisfies the following:
\begin{enumerate}
\item The reduction mod $\pi$ of $\psi$ is the $\xqb$-power Frobenius relative to $u$, that is,
$\psi(x) \equiv u(x)^\xqb \bmod \pi C$.

\item The restriction of $\psi$ to $R$ coincides with the fixed $\phi$ on $R$, 
that is, the following diagram commutes
	$$\xymatrix{
	B \ar[r]^\psi & C \\
	R \ar[r]_\phi \ar[u] & R \ar[u] 
	}$$
\end{enumerate}
A \emph{$\pi$-derivation} $\d$ from $B$ to $C$ means a set-theoretic map
$\d:B \map C$ satisfying the following for all $x,y \in B$
	\begin{eqnarray*}
	\label{der}
	\d(x+y) &=& \d (x) + \d (y)  \\
	\d(xy) &=& u(x)^\xqb \d (y) +  \d (x) u(y)^\xqb + \pi \d (x) \d (y)
	\end{eqnarray*}
such that for all $r\in R$, we have
$$
\d(r) = \frac{\phi(r)-r^\xqb}{\pi}.
$$
When $C=B$ and $u$ is the identity map, we will call this simply a $\pi$-derivation on $B$.

It follows that the map $\phi: B \map C$ defined as 
	$$
	\phi(x) := u(x)^\xqb + \pi \d (x)
	$$ 
is a lift of Frobenius in the sense above. On the other hand, 
for any flat $R$-algebra $B$ with a lift of Frobenius $\phi$, one can define
the $\pi$-derivation $\d(x)= \frac{\phi(x)-x^\xqb}{\pi}$ for all $x \in B$.

Note that this definition depends on the choice of uniformizer $\pi$, but in a transparent way:
if $\pi'$ is another uniformizer, then $\d(x)\pi/\pi'$ is a $\pi'$-derivation. This correspondence
induces a bijection between $\pi$-derivations $B\to C$ and $\pi'$-derivations $B\to C$.

\subsection{Witt vectors}
We will present three different points of view on function-field Witt vectors, all parallel to the mixed
characteristic case. But there is perhaps one unfamiliar element below, which is that we will work relative to
our general base $R$, and it already has a lift of Frobenius. The consequence is that we need to pay attention
to certain twists of the scalars by Frobenius, which are invisible over the absolute base $R=\hA$. However this
unfamiliar element has nothing to do with the difference between mixed and equal characteristic and only with
the difference between the relative and the absolute setting.

Let $B$ be an $R$-algebra with structure map $u:R\to B$.

(1) The ring $W(B)$ of \emph{$\pi$-typical Witt vectors} can be defined as
the unique (up to unique isomorphism) $R$-algebra $W(B)$ with a $\pi$-derivation $\d$ on 
$W(B)$ and an $R$-algebra homomorphism $W(B) \map B$ such that, given any 
$R$-algebra $C$ with a $\pi$-derivation $\d$ on it and an $R$-algebra map 
$f:C \map B$, there exists a unique $R$-algebra homomorphism $g:C \map W(B)$  such that the diagram
	$$
	\xymatrix{
	W(B) \ar[d] & \\
	B & C \ar[l]_f \ar[ul]_g 
	}
	$$
commutes and $g \circ \d = \d \circ g$.
Thus $W$ is the right adjoint of the forgetful functor from $R$-algebras with $\pi$-derivation
to $R$-algebras. For details, see section 1 of~\cite{bor11a}.
This approach follows that of \cite{Jo} to the usual $p$-typical Witt vectors.

(2) If we restrict to flat $R$-algebras $B$, then we can ignore the concept of $\pi$-derivation
and define $W(B)$ simply by expressing the universal property above
in terms of Frobenius lifts, as follows.
Given a flat $R$-algebra $B$, the ring $W(B)$ is
the unique (up to unique isomorphism) flat $R$-algebra $W(B)$ with a lift of Frobenius (in the sense
above) $F:W(B) \map W(B)$ and an $R$-algbebra homomorphism $W(B) \map B$ such that
for any flat $R$-algebra $C$ with a lift of Frobenius $\phi$ on it and an $R$-algebra map $f:C \map B$, 
there exists a unique $R$-algebra homomorphism $g:C \map W(B)$ such that the diagram
	$$
	\xymatrix{
	W(B) \ar[d] & \\
	B & C \ar[l]_f \ar[ul]_g 
	}
	$$
commutes and $g \circ \phi = F \circ g$.

(3) Finally, returning to the case of general $R$-algebras $B$, one can also define Witt vectors in terms of 
the Witt polynomials.
For each $n \geq 0$ let us define $B^{\phi^n}$ to be the $R$-algebra
with structure map $R \stk{\phi^n} {\map} R \stk{u}{\map} B$ and define the \emph{ghost rings}
to be the product $R$-algebras
$\Pi^n_\phi B = B \times B^\phi \times \cdots  \times B^{\phi^n}$ and 
$\Pi_\phi^\infty B= B \times B^\phi \times \cdots$. 
Then for all $n \geq 1$ there exists a \emph{restriction}, or \emph{truncation},
map $T_w:\Pi_\phi^nB \map \Pi_\phi^{n-1}B$ given by $T_w(w_0,\cdots,w_n)= (w_0,\cdots,w_{n-1})$.
We also have the left shift \emph{Frobenius} operators $F_w:\Pi_\phi^n B \map \Pi_\phi^{n-1} B$ given by 
$F_w(w_0,\dots,w_n) = (w_1,\dots,w_n)$. Note that $T_w$ is an $R$-algebra morphism, but
$F_w$ lies over the Frobenius endomorphism $\phi$ of $R$.

Now as sets define
	\begin{equation}
	\label{eq-witt-coord}
	W_n(B)=B^{n+1},	
	\end{equation}
and define the set map $w:W_n(B) \map \Pi_\phi^n B$  by $w(x_0,\dots,x_n)= (w_0,\dots,w_n)$ where 
	\begin{align}
	\label{eq-witt-poly}
	w_i = x_0^{\xqb^i}+ \pi x_1^{\xqb^{i-1}}+ \cdots + \pi^i x_i
	\end{align}
are the \emph{Witt polynomials}.
The map $w$ is known as the {\it ghost} map. (Do note that under the traditional indexing, 
used in many sources going back to Witt~\cite{Witt:Vectors}, our $W_n$ would be 
denoted $W_{n+1}$.) We can then define the ring $W_n(B)$, the ring 
of truncated $\pi$-typical Witt vectors, by the following theorem as in the $p$-typical case~\cite{hessl05},
proposition 1.2:

\begin{theorem}
\label{wittdef}
For each $n \geq 0$, there exists a unique functorial $R$-algebra structure on $W_n(B)$ such that
$w$ becomes a natural transformation of functors of $R$-algebras.
\end{theorem}

Note that, unlike with the usual Witt vectors in mixed characteristic, addition for function-field Witt vectors
is performed componentwise. This is because the Witt polynomials~(\ref{eq-witt-poly}) are additive. This might
appear to defeat the whole point of Witt vectors and arithmetic jet spaces. But this is not so. The reason is
that while the additive structure is the componentwise one, the $A$-module structure is not. So the difference
is only that, unlike in mixed characteristic where $A=\bZ$, a group structure is weaker than $A$-module
structure. In fact, because the Witt polynomials are $k$-linear, the $k$-vector space structure on
$W_n(B)$ is the componentwise one. This is just like with the $p$-typical Witt vectors, where multiplication by
roots of $x^p-x$ can be performed componentwise.

For the convenience of the reader, we give some examples the proofs of which we leave as exercises.
If the structure map $A\to B$ factors through $A/\fp$ and $B$ is perfect, then multiplication is given by the formula
	$$
	(x_0,x_1,\dots)\cdot (y_0,y_1,\dots) = (z_0,z_1,\dots), \quad 
	\text{where} \quad z_n = \sum_{i+j=n}x_i^{\xqb^j}y_j^{\xqb^i}.
	$$
For example, if $B=R=A/\fp=\bF_{\xqb}$, then $W(B)$ is identified with the power-series ring $B[[\pi]]$,
where $\pi$ corresponds to the Witt vector $(0,1,0,0,\dots)$.
At the opposite extreme, where $\pi$ is invertible in $B$, the ghost map is an isomorphism.
So $W(B)$ is isomorphic to the product ring $B\times B\times\cdots$ and not a power-series
ring.

\subsection{Operations on Witt vectors}
\label{subsec-witt-operations}
Now we recall some important operators on the Witt vectors. There are the 
\emph{restriction}, or \emph{truncation}, maps $T:W_n(B) \map W_{n-1}(B)$ given by $T(x_0,\dots,x_n) = 
(x_0,\dots, x_{n-1})$. Note that $W(B) = \varprojlim W_n(B)$.
There is also the {\it Frobenius} ring homomorphism 
$F:W_n(B) \map W_{n-1}(B)$, which can be described in terms of the ghost map. 
It is the unique map which is functorial in $B$ and makes the 
following diagram commutative
	\begin{align}
	\xymatrix{
	W_n(B) \ar[r]^w \ar[d]_F & \Pi^n_\phi B \ar[d]^{F_w} \\
	W_{n-1}(B) \ar[r]_-w & \Pi_\phi^{n-1} B^n
	} \label{F}
	\end{align}
As with the ghost components, $T$ is an $R$-algebra map but $F$ lies over the Frobenius endomorphism $\phi$
of $R$.

Next we have the {\it Verschiebung} $V:W_{n-1}(B) \map W_n(B)$ given by 
$$V(x_0,\dots,x_{n-1}) = (0,x_0,\dots,x_{n-1}).$$ Let 
$V_w:\Pi_\phi^{n-1}B \map \Pi_\phi^n B$ be the additive map given by 
$$V_w(w_0,..,w_{n-1})= (0, \pi w_0,\dots,\pi w_{n-1}).$$ Then the Verschiebung
$V$ makes the following diagram commute:
	\begin{align}
	\xymatrix{
	W_{n-1}(B) \ar[r]^-w \ar[d]_V & \Pi_\phi^{n-1}B \ar[d]^{V_w} \\
	W_n(B) \ar[r]_-w & \Pi_\phi^nB
	}\label{V}
	\end{align}
For all $n \geq 0$ the Frobenius and the Verschiebung satisfy the identity
	\begin{align}
	\label{FV-pi}
	FV(x) = \pi x.
	\end{align}
The Verschiebung is not a ring homomorphism, but it is $k$-linear.

Finally, we have the multiplicative Teichm\"uller map $[~]:B \map W_n(B)$ given by 
$x\mapsto [x]= (x,0,0,\dots)$. Here in the function-field setting, $[~]$ is additive and even a 
homomorphism of $k$-algebras but is not a homomorphism of $A$-algebras. This can be compared to
the mixed-characteristic setting, where
it is a homomorphism of monoids but not a homomorphism of $\bZ$-algebras.

\subsection{Computing the universal map to Witt vectors}
Given an $R$-algebra $C$ with a $\pi$-derivation $\d:C\map C$ and an
$R$-algebra map $f:C \map B$, we will now describe the universal lift $g:C \map W(B)$. 
The explicit description of $g$ leads us to proposition \ref{coordinate} which
is used in section~\ref{sec-computation2} in computations for Drinfeld modules of rank
$2$. The reader may skip this subsection without breaking continuity till then.

It is enough to work in the case where both $B$ and $C$ are flat over $R$. 
Then the ghost map $w:W(B) \map \Pi_\phi^\infty B$ is injective.
Consider the map $[\phi]: C \map \Pi_\phi^\infty C$ given by $x \mapsto 
(x,\phi(x), \phi^2(x),\dots)$. 
Then we have the following commutative diagram:
	$$\xymatrix{
	& & C \ar[ld]_{f\circ [\phi]} \ar[d]^{[\phi]}\ar@/_1pc/[lld]_g \\
	W(B) \ar[r]^w \ar[d]^F & \Pi_\phi^\infty B \ar[d]^{F_w} & \Pi_\phi^\infty C 
	\ar[l]_f \ar[d]^{F_w}\\
	W(B) \ar[r]^w & \Pi_\phi^\infty B & \Pi_\phi^\infty C \ar[l]_f
	}$$
Thus the map $f\circ [\phi]:C \map \Pi_\phi^\infty B$ factors through
$W(B)$ as our universal map $g:C \map W(B)$. 

Let us now give an inductive description of the map $g$. Write 
	$$
	g(x)= (x_0,x_1,\cdots) \in W(B).
	$$ 
Then from the above diagram $w\circ g = f\circ [\phi]$. Therefore
the vector $(x_0,x_1,\dots)$ is the unique solution to the system of equations
	\begin{equation}
	\label{Ps}
	x_0^{\xqb^n}+ \pi x_1^{\xqb^{n-1}}+ \cdots + \pi^n x_n = f(\phi^n(x)),
	\end{equation}
for $n\geq 0$. For example, we have  $x_0 = f(x)$ and $x_1 = f(\d(x))$. 

Now consider the case where $B$ itself has a $\pi$-derivation,
$C=B$, and $f= \mathbbm{1}$. 
For any $x \in B$, let us write $x^{(n)} := \d^n(x)$, or simply $x'=\d(x)$,
$x''=\d^2(x)$ and so on.

\begin{proposition}
\label{coordinate}
We have $x_0= x$, $x_1 = x'$ and $x_2 = x'' + \pi^{\xqb-2}(x')^\xqb.$
\end{proposition}

\begin{proof}
As stated above, equalities $x_0=x$ and $x_1=x'$ follow immediately from~(\ref{Ps}).
For $n=2$, we have
	\begin{eqnarray*}
	x_0^{\xqb^2}+ \pi x_1^\xqb + \pi^2 x_2 &=& \phi^2(x) \\
	&=& \phi(x^\xqb+ \pi x') \\
	&=& \phi(x)^\xqb + \pi \phi(x') \\
	&=& x^{\xqb^2} + \pi^\xqb (x')^\xqb + \pi((x')^\xqb + \pi x'')
	\end{eqnarray*}
And therefore we have $x_2 = x''+ \pi^{\xqb-2} (x')^\xqb$.
\end{proof}

\section{$A$-module schemes, Jet Spaces and prelimineries}

An $A$-module scheme over $S= \Spf R$ is by definition a pair $(E,\vp_E)$, where $E$ is a commutative group
object in the category of $S$-schemes and $\vp_E:A \map \End(E/S)$ is a ring map. (Here and below, by a scheme
over the formal scheme $S$, we mean a formal scheme formed from a compatible family of schemes over the schemes
$\Spec R/\fp^n R$.) Then the tangent space $T_0E$ at the identity has two $A$-modules structures: one coming by
restriction of the usual $R$-module structure to $A$, and the other coming from differentiating $\vp_E$. We
will say that $(E,\vp_E)$ is {\it strict} if these two $A$-module structures coincide, that is if
the composition 
	$$
	A \map \End(E/S) \map \End_R(T_0E)
	$$
agrees with the composition
	$$
	A \longlabelmap{\theta} R \longmap \End_R(T_0E).
	$$
We say it is \emph{admissible} if
it is both strict and isomorphic to the additive group $\hG=\hat{\mathbb{G}}_{\mathrm{a}/S}$ as a group scheme.
 
We will denote this induced map to tangent space as $\theta: A \map R$.
(Note that it is best practice to require only the isomorphism with $\hG$ to exist locally on $S$. So below, our
Drinfeld modules would more properly be called \emph{coordinatized} Drinfeld modules.)

A {\it Drinfeld module $(E,\vp_E)$ of rank $r$} is an admissible $A$-module scheme over $S$ such that for each
non-zero $a \in A$, the group scheme $\ker(\vp_E(a))$ is finite flat of degree
$|a|^r=\xqa^{-r\mathrm{ord}_\infty(a)}$ over $S$. (See~\cite{Ge1}, (1.4), or~\cite{Laumon-book-vol1}, p.\ 4.)

\begin{proposition}
\label{restricted}
Let $f$ be an endomorphism of the $\bF_\xqa$-module scheme $\hat{\mathbb{G}}_{\mathrm{a}/S}$ over $S$. Then
given any coordinate $x$ on $E$, the map $f$ is of the form 
	$$
	f(x)= \sum_{i=0}^\infty a_i x^{\xqa^i},
	$$
where $f$ is a restricted power series, meaning $a_i \map 0$ $\pi$-adically as $i \map \infty$.
\end{proposition}

\begin{proof}
Let $f \in \Hom(\hG,\hG)$ be an additive endomorphism of $\hG$. Then $f$ 
is given a restricted power series $\sum_i b_i x^i$ such that $b_i \map 0$ as 
$i \map \infty$. Since $f$ is additive, we have $b_i=0$ unless $i$ is a power of $p$. 
Second, because $f$ is $\bF_\xqa$-linear, we have $\sum_i b_{p^i}(cx)^{p^i}=c\sum_i b_{p^i}x^{p^i}$
for all $c\in\bF_\xqa$. Considering the case where $c$ is a generator of $\bF_\xqa^*$,
we see this implies $b_{p^i}=0$ unless $p^i$ is a power of $\xqa$.
\end{proof}

Let $R\{\tau\}\h$ be the subring of $R\{\{\tau\}\}$ consisting of (twisted) restricted power series. 
Then by proposition \ref{restricted}, the $\bF_\xqa$-linear
morphisms between two admissible $A$-module schemes $E_1$ and $E_2$ over $\Spf R$ are given in coordinates by 
elements in $R\{\tau\}\h$ where $\tau$ acts as $\tau(x) =x^\xqa$:
\begin{equation}
	\label{eq-tau}
	\Hom_{\bF_\xqa}(E_1,E_2)=R\{\tau\}\h.
\end{equation}

\subsection{Prolongation sequences and jet spaces}
Let $X$ and $Y$ be schemes over $S=\Spf R$. We say a pair $(u,\d)$ is a {\it prolongation}, and write 
$Y \stk{(u,\d)}{\map} X$, if $u: Y \map X$ is a map of schemes over $S$ and $\d: \Ou_X \map u_*\Ou_Y$ is a 
$\pi$-derivation making the following diagram commute: 
	$$
	\xymatrix{
	R \ar[r] &  u_* \Ou_Y \\
	R \ar[u]^\d \ar[r] &  \Ou_X \ar[u]_\d \\
	} 
	$$ 
Following \cite{Bui3}, a {\it prolongation sequence} is a sequence of prolongations
	$$
	\xymatrix{
	\Spf R & T^0 \ar_-{(u,\d)}[l] & T^1 \ar_-{(u,\d)}[l] & \cdots\ar_-{(u,\d)}[l]},
	$$
where each $T^n$ is a scheme over $S$. We will often use the notation $T^*$ or $\{T_n\}_{n \geq 0}$.
Note that if the  $T^n$ are flat over $\Spf R$ then having a 
$\pi$-derivation $\d$ is equivalent to having lifts of Frobenius $\phi:T^{n+1}\to T^n$.

Prolongation sequences form a category $\mcal{C}_{S^*}$, where a morphism $f:T^*\to U^*$ is 
a family of morphisms $f^n:T^n\to U^n$ commuting with both the $u$ and $\d$, in the evident sense.
This category has a final object $S^*$ given by $S^n=\Spf R$ for all $n$, where each $u$ is the identity and
each $\d$ is the given $\pi$-derivation on $R$.

For any scheme $Y$ over $S$, for all $n \geq 0$ we define the $n$-th jet space $J^nX$ (relative to $S$) as 
	$$
	J^nX (Y) := \Hom_S(W_n^*(Y),X)
	$$
where $W_n^*(Y)$ is defined in section 10.3 of \cite{Bo2}.
We will not define $W_n^*(Y)$ in full generality here. Instead,
we will define $\Hom_S(W_n^*(Y),X)$ in the affine case, and that will be sufficient for the purposes of this
paper. Write $X = \Spf C$ and $Y=\Spf B$. Then $W_n^*(Y)= \Spf W_n(B)$ and so $J^nX (B)$
is the set of $R$-algebra homomorphisms $C\to W_n(B)$:
	\begin{equation}
		\label{eq-jet-witt}
		J^nX (B)=\Hom_R(C,W_n(B)).
	\end{equation}

Then $J^*X:= \{J^nX \}_{n \geq 0}$ forms a prolongation sequence,
called the {\it canonical prolongation sequence}. As in the mixed-characteristic
case (\cite{Bui3}, proposition (1.1)), $J^*X$ satisfies the following 
universal property---for any $T^* \in \mcal{C}_{S^*}$ and $X$ a scheme over 
$S^0$, we have
	$$
	\Hom(T^0,X) = \Hom_{\mcal{C}_{S^*}}(T^*, J^*X)
	$$

Let $X$ be a scheme over 
$S= \Spf R$. Define $X^{\phi^n}$ by $X^{\phi^n}(B) := X(B^{\phi^n})$ for any $R$-algebra $B$. 
In other words, $X^{\phi^n}$ is $X \times_{S,\phi^n} S$, the pull-back of $X$ under the map $\phi^n:S\to S$.
Next define 
	$$
	\Pi^n_\phi X= X \times_S X^\phi \times_S \cdots \times_S X^{\phi^n}.
	$$ 
Then for any $R$-algebra $B$ we have $X(\Pi_\phi^n B) = X(B)\times_S \cdots \times_S X^{\phi^n}(B)$.
Thus the ghost map $w$ in theorem \ref{wittdef} defines a map of $S$-schemes 
	$$
	w:J^nX \map \Pi_\phi^nX.
	$$ 
Note that $w$ is injective when evaluated on points with coordinates in any flat $R$-algebra.

The operators $F$ and $F_w$ in (\ref{F}) induce maps $\phi$ and $\phi_w$ 
as follows
\begin{align}
	\xymatrix{
	J^nX \ar[r]^w\ar[d]_\phi & \Pi_\phi^n X \ar[d]^{\phi_w}\\
	J^{n-1}X \ar[r]_w & \Pi_\phi^{n-1} X
	}
	\label{phiw}
\end{align}
where $\phi_w$ is the left-shift operator given by
	$$
	\phi_w(w_0,\dots,w_n)= (\phi_S(w_1),\dots,\phi_S(w_n)),
	$$
and where  $\phi_S:X^{\phi^i} \map X^{\phi^{i-1}}$ is the composition given in the following diagram:
\begin{align}
	\label{ko}
	\xymatrix{
	X^{\phi^i}\ar[r]^-{\sim} & X^{\phi^{i-1}} \times_{S,\phi} S \ar[d] 
	\ar[r]^-{} & 
	X^{\phi^{i-1}} \ar[d] \\
	& S \ar[r]_\phi & S.
	}
\end{align}

Now let $E$ be an $A$-module scheme over $S$ with action map $A \stk{\vp_E}\to \End_S(E)$.
Then the functor it represents takes values in $A$-modules, and hence so does
the functor $B\mapsto E(W_n(B))$.
In this way, for each $n\geq 0$, the $S$-scheme $J^n E$ comes with an $A$-module structure.
We denote it by $\vp_{J^nE}:A \map \End_S(J^nE)$. Similarly, 
$\vp_E$ induces an $A$-linear structure $\vp_{E^{\phi^n}}$ on each
$E^{\phi^n}$. In this case, it is easy to describe explicitly. It is the componentwise one:
	$$
	\vp_{\Pi_\phi^n E}(w_0,\dots,w_n)= (\vp_E(w_0),\dots,\vp_{E^{\phi^n}}(w_n)).
	$$
The ghost map $w:J^nE\to \Pi_\phi^n E$ and the truncation map $u:J^nE\to J^{n-1}E$
homomorphisms of $A$-module schemes over $S$. 
This is because they are given by applying the $A$-module scheme $E$ to the $R$-algebra maps
$w:W_n(B)\to \Pi^n_\phi B$ and $T:W_n(B)\to W_{n-1}(B)$. On the other hand,
the Frobenius map $\phi:J^nE\to J^{n-1}E$ is a homomorphisms of $A$-module schemes lying over 
the Frobenius endomorphism $\phi$ of $S$. In other words, the induced map $J^nE\to (J^{n-1}E)^\phi$
is a homomorphism of $A$-module schemes over $S$.

\subsection{Coordinates on jet spaces}
Given an isomorphism of $S$-schemes $E\to\hG$, we have an induced bijection, by~(\ref{eq-jet-witt}),
\begin{equation}
\label{eq-jet-Witt-again}
(J^nE)(B)\isomap W_n(B).
\end{equation}
Now recall the bijection $W_n(B)\isomap B^{n+1}$ of equation~(\ref{eq-witt-coord}).
Combining the two, we see that given a coordinate 
$x$ on an admissible $A$-module scheme $E$, we have a canonical system of coordinates 
$(x_0,\dots,x_n)$ on $J^nE$. We will use these \emph{Witt coordinates} without further comment.
We emphasize once again that there are other canonical systems of coordinates on $J^nE$,
for instance the \emph{Buium--Joyal} coordinates denoted $x,x',x'',\dots$.
They are related by the formulas of proposition~\ref{coordinate}. Each has their own advantages.

We will now describe the above maps explicitly in the Buium--Joyal coordinates.
Let $\Ou(E)= R[x]\h$. Then for each $n$, $\Ou(J^nE)= R[x,x',\dots , x^{(n)}]$.
Then for each $n$, the corresponding algebra maps $u^*$ and $\phi^*$ from
$\Ou(J^nE) \map \Ou(J^{n+1}E)$ are given as follows:
\begin{align}
	u^*(x^{(i)}) &= x^{(i)}, \text{ for all } i, \nonumber \\
	\phi^*(x^{(i)}) &= (x^{(i)})^{\xqb} + \pi x^{(i+1)}, \text{ for all } i.
	\label{eq:phi-man}
\end{align}

\subsection{Character groups}
\label{sec:char-groups}
Let $\hG$ denote the additive group over $S$, i.e., the formal spectrum of the $\pi$-adic completion of $R[x]$,
with the tautological $A$-module structure $\vp_{\hG}$
given by the usual multiplication of scalars: $\vp_{\hG}(a) = a\tau^0$.
We will maintain this convention throughout the paper.

Given a prolongation sequence $T^*$ we can define its shift $T^{*+n}$ by 
$(T^{*+n})^j:= T^{n+j}$ for all $j$ (as in \cite{Bui3}, p.\ 106).
	$$
	\Spf R \stk{(u,\d)}{\leftarrow} T^n \stk{(u,\d)}{\leftarrow} T^{n+1}\dots 
	$$
We define a {\it $\d$-morphism of order $n$} from $X$ to $Y$ to be a 
morphism $J^{*+n}X \map J^*Y$ of prolongation sequences.
We define a {\it character of order $n$}, $\Theta:(E,\vp_E) \map (\hG,\vp_{\hG})$ 
to be a $\d$-morphism of order $n$ from $E$ to $\hG$
which is also a homomorphism of $A$-module objects.
By the same argument as in the mixed characteristic case (proposition (1.9) of~\cite{Bui3}),
an order $n$ character is equivalent to a homomorphism
$\Theta:J^nE \map \hG$ of $A$-module schemes over $S$. We denote the group of 
characters of order $n$ by $\bX_n(E)$. So we have
	$$
	\bX_n(E)=\HomA(J^nE,\hG),
	$$
which one could take as an alternative definition. Note that $\bX_n(E)$ comes with an
$R$-module structure since $\hG$ is an $R$-module scheme over $S$. Also the inverse system 
$J^{n+1}E \stk{u}{\map} J^nE$ defines a directed system 
	$$
	\bX_n(E) \stk{u^*}{\map} \bX_{n+1}(E) \stk{u^*}{\map}\cdots
	$$
via pull back. Each morphism $u^*$ is injective because each $u$ has a section
(typically not $A$-linear).
We then define $\bX_\infty(E)$ to be the $R$-module direct limit $\varinjlim \bX_n(E)$.

Similarly, pre-composing with the Frobenius map $\phi:J^{n+1}E\to J^nE$ induces a Frobenius operator
$\phi:\bX^n(E)\to \bX^{n+1}(E)$. However since $\phi:J^{n+1}E\to J^nE$ is not a morphism over $\Spf R$ but
instead lies over the Frobenius endomorphism $\phi$ of $\Spf R$, some care is required.
Consider the relative Frobenius morphism $\phi_{E/R}$, defined to be the unique
morphism making the following diagram commute:
	$$
	\xymatrix{
	J^{n+1}E \ar@{.>}^{\phi_{E/R}}[rd] \ar@/^/[rrd]^\phi \ar@/_/[ddr]& & \\
	& J^nE \times_{(\Spf R),\phi} \Spf R \ar[d] \ar[r] & J^nE \ar[d] \\
	& \Spf R \ar[r]_\phi & \Spf R 
	}
	$$
Then $\phi_{E/R}$ is a morphism of $A$-module formal schemes over $\Spf R$.
Now given a $\d$-character $\Theta:J^nE\to \hG$, define $\phi^*\Theta$ to be 
the composition
\begin{equation}
	J^{n+1}E \longlabelmap{\phi_{E/R}} J^nE \times_{(\Spf R),\phi} \Spf R 
	\longlabelmap{\Theta\times\mathbbm{1}} \hG\times_{(\Spf R),\phi} \Spf R\longlabelmap{\iota} \hG
\end{equation}
where $\iota$ is the isomorphism of $A$-module  schemes over $S$
coming from the fact that $\hG$ descends to $\hA$ as an $A$-module scheme. For any $R$-algebra $B$,
the induced morphism on $B$-points is
	$$
	E(W_{n+1}(B)) \longlabelmap{E(F)} E(W_n(B)^\phi) \longlabelmap{\Theta_B^\phi} 
	B^\phi \longlabelmap{b\mapsto b} B.
	$$
Note that this composition $E(W_{n+1}(B))\to B$ is indeed a morphism  of $A$-modules because
identity map $B^\phi\to B$ is $A$-linear, which is true because $\phi$ restricted to $\hA$
is the identity.

Thus we have an additive map $\bX_n(E) \to \bX_{n+1}(E)$ given by $\Theta\mapsto \phi^*\Theta$. Note
that this map is not $R$-linear. However, the map
	$$
	\phi^*:\bX_n(E) \longmap \bX_{n+1}(E)^\phi, \quad \Theta\mapsto \phi^*\Theta 
	$$ 
is $R$-linear, where $\bX_{n+1}(E)^\phi$ denotes the abelian group $\bX_{n+1}(E)$ with 
$R$-module structure defined by the law $r\cdot \Theta := \phi(r)\Theta$.
Taking direct limits in $n$, we obtain an $R$-linear map
	$$ 
	\bX_\infty(E) \longmap \bX_\infty(E)^\phi, \quad \Theta\mapsto\phi^*\Theta.  
	$$
In this way, $\bX_\infty(E)$ is a left module over the twisted polynomial ring $R\{\phi^*\}$ with
commutation law $\phi^*r = \phi(r)\phi^*$.

\section{Admissible modules}

Let $(E,\varphi_E)$ be an admissible $A$-module scheme over $S=\Spf R$.
By equation~(\ref{eq-tau}), we can write
	\begin{equation}
		\label{eq-phi}
	\varphi_E(t) = \sum a_i \tau^i
	\end{equation}
with $a_i\in R$, $a_i\to 0$, and $a_0 = \pi=\theta(t)$. For brevity, we will typically write the pair
$(E,\varphi_E)$ as $E$. We remind the reader that
$\hG$ implicitly has the tautological $A$-module structure defined in
section~\ref{sec:char-groups}.

The main purpose of this section is to establish some facts that will be used in the proof
of theorem~\ref{kernel} below.
We emphasize that in this application $E$ will not be a Drinfeld module.

\begin{proposition}\label{pro:tangent}
Any $A$-linear morphism $f:E\to G$ between admissible $A$-modules is determined by the induced morphism
on tangent spaces. More precisely,
if we write $\vp_E(t)=\pi\tau^0+\sum_{j\geq 1}a_j \tau^j$, $\vp_G(t)=\pi\tau^0+\sum_{j\geq 1}c_j \tau^j$,
and $f=\sum_{i} b_i\tau^i$, then $f$ is determined by $b_0$, as follows:
	$$
	b_r = \frac{1}{\pi-\pi^{\xqa^r}}\sum_{i=0}^{r-1} (b_ia_{r-i}^{\xqa^i}-c_{r-i}b_i^{\xqa^{r-i}}).
	$$
\end{proposition}
\begin{proof} 
Because $f$ is $B$-linear, we have
	$$
	\big(\sum_{i\geq 0} b_i \tau^i\big)\big(\pi\tau^0+\sum_{j\geq 1}a_j \tau^j\big) 
	= \big(\pi\tau^0+\sum_{j\geq 1}c_j \tau^j\big)\big(\sum_{i\geq 0} b_i \tau^i\big).
	$$
Comparing the coefficients of $\tau^r$, we have
	$$
	b_0 a_r^{\xqa^0}+\cdots+b_{r-1}a_1^{\xqa^{r-1}} + b_r \pi^{\xqa^r} 
	= \pi b_r + c_1 b_{r-1}^{\xqa}+\cdots+c_r b_0^{\xqa^r}.
	$$
Therefore we have
	$$
	b_r(\pi-\pi^{\xqa^r}) = \sum_{i=0}^{r-1} (b_ia_{r-i}^{\xqa^i}-c_{r-i}b_i^{\xqa^{r-i}}).
	$$
Since $R$ is $\pi$-torsion free and $1-\pi^{\xqa^r-1}$ is invertible for $r\geq 1$,
this determines each $b_r$ uniquely in terms of $b_0,\dots,b_{r-1}$. Therefore $b_0$ determines
each $b_r$.
\end{proof}


\begin{corollary}
\label{cor-GG}
The $R$-module map $R\to \HomA(\hG,\hG)$ defined by $b\mapsto b\tau^0$ is an 
isomorphism.
\end{corollary}

Now consider the subset $S^\dagger \subset R\{\tau\}\h$ defined by 
	\begin{align}
	S^\dagger := \big\{ \sum_{i\geq 0} b_i \tau^i \in R\{\tau\}\h
		\mid v(b_i) \geq i,\mb{ for all }i \mb{ and } b_0 \in R^*\big\}.
	\end{align}
Here, and below, we write $v(b)$ for the minimal $i$ such that $b\in\fp^i R$.
(Note that $v$ may not be a valuation if $R$ is not a discrete valuation ring.)


\begin{proposition}
\label{pal2}
$S^\dagger$ is a group under composition.
\end{proposition}

Note that a similar group of automorphisms appears in Dupuy~\cite{du}, section 4.3.

\begin{proof}
The fact that $S^\dagger$ is a submonoid of $R\{\tau\}\h$ under composition follows immediately
from the law $b\tau^i\circ c\tau^j = bc^{\xqa^i}\tau^{i+j}$ and linearity. Indeed if
$v(b)\geq i$ and $v(c)\geq j$, then $v(bc^{\xqa^i}) \geq i+j$.

Now let us show that any element $f = \sum b_i \tau^i \in S^\dagger$ has an inverse under composition. Let 
$g =\sum_{n=0}^\infty c_n \tau^n$, where $c_0 = b_0^{-1}$ and we define inductively 
$c_n = - b_0^{-\xqa^n}(c_0b_n +c_1 b_{n-1}^\xqa + \dots + c_{n-1} b_1^{\xqa^{n-1}})$. 
Then it is easy to check that $g \circ f = \mathbbm{1}$.
Take $n\geq 1$ and assume $v(c_i) \geq i$ for all $i=0,\dots ,n-1$. Then it is enough to show $v(c_n) \geq n$.
We have $v(c_n) \geq \min\{v(c_ib_{n-i}^{\xqa^i})\,|\, i=0,\dots , n-1\}$. Now
	\begin{eqnarray*}
	v(c_ib_{n-i}^{\xqa^i}) &=& v(c_i) + \xqa^i v(b_{n-i}) \\
	&=& i + \xqa^i(n-i) \\
	& \geq & i+ (n-i) = n.
	\end{eqnarray*}
Therefore the left inverse $g$ of $f$ lies in $S^\dagger$.
 
Now consider $g'= \sum_{n=0}^\infty d_n \tau^n \in R\{\{\tau\}\}$, where $d_0 = b_0^{-1}$ and we inductively 
define $d_n= -b_0^{-1}(b_1 d_{n-1}^{\xqa^1} + b_2 d_{n-2}^{\xqa^2} +\dots  + b_n d_0^{\xqa^n})$. 
Then as above, one can easily check that $f \circ g' = \mathbbm{1}$ and hence it is a right inverse of $f$ in
$R\{\{\tau\}\}$. But using the associativity property of $R\{\{\tau\}\}$
we get $g'=(g\circ f) \circ g' = g \circ (f\circ g')=g$ and hence $g$ is both a left and
right inverse of $f$ in $S^\dagger$.
\end{proof}

\begin{proposition}\label{pro:formal2}
Let $B$ denote the subring $\bF_q[t]\subseteq A$.
Let $f:E\to G$ be a $B$-linear homomorphism of admissible $A$-module schemes over $\Spf R$.
Then $f$ is $A$-linear.
\end{proposition}
\begin{proof}
Given any element $a\in A$, we will show $\vp_G(a)\circ f = f\circ \vp_E(a)$.
Both sides are $B$-linear homomorphisms $E\to G$; indeed, $f$ is $B$-linear by assumption,
and both $\vp_G(a)$ and $\vp_E(a)$ are $B$-linear because $A$ is commutative.
Furthermore, on tangent spaces, $\vp_G(a)\circ f$ is multiplication by $af'(0)$,
and $f\circ \vp_E(a)$ is multiplication by $f'(0)a$; this is because the $A$-module schemes are admissible.
Thus the two morphisms agree on tangent spaces and therefore they agree, by proposition~\ref{pro:tangent}.
\end{proof}

In other words, the forgetful functor from admissible $A$-modules schemes over $R$ to admissible
$B$-module schemes over $R$ is fully faithful. This remains true if we allow $B$ to be not just $\bF_q[t]$
but any sub-$\bF_q$-algebra of $A$ strictly containing $\bF_q$.

\begin{lemma}
\label{palha}
If $\xqa \geq 3$, then $\xqa^i - \xqa^{i-j} -j-1 \geq 0$ for all $j=1,\dots ,i$.
\end{lemma}

\begin{proof}	
Consider $f(x)= \xqa^i - \xqa^{i-x} -x-1$, for $1\leq x \leq i$. 
Then $f(1) \geq 0$ since $\xqa \geq 3$. Now $f'(x) = \xqa^{i-x}\ln{\xqa} -1 $. 
Since $\ln{\xqa} > 1$ for $\xqa \geq 3$, we have $f'(x) \geq 0$ for all 
$1 \leq x \leq i$ and hence $f(x) \geq 0$ for all $1 \leq x \leq i$ and we are 
done. 
\end{proof}

\begin{lemma}
\label{dod}
For $q=2$ and $i \geq 2$, $q^i-i-1 \geq 1$.
\end{lemma}
\begin{proof}
Consider the function $h(x)= q^x-x$ for $x \geq 2$.
Then $h'(x)=q^x \ln{q}  -1= \ln{q^{q^x}} -1> 0$ since $x \geq 2$. Therefore $h$
is a strictly increasing function and hence the minimum is attained at i=2.
Therefore $q^i-i \geq q^2-2= 2$ and the result follows. 
\end{proof}

\begin{lemma}
\label{q2}
For $q=2$ and $i \geq 2$ and $j=1,\dots , i$
$$q^i- q^{i-j}-j \geq 1$$
\end{lemma}
\begin{proof}
For $j=i$, the result follows from lemma \ref{dod}. So we may
assume $1 \leq j \leq i-1$.
Let $H(x):= q^i-q^{i-x}-x$ where $1 \leq x \leq i-1$. Then
$H'(x)= q^{i-x} \ln{q} -1$. Since $x \leq i-1$ implies $i-x \geq 1$ and
hence $q^{i-x} \geq q$. Therefore we get
\begin{eqnarray*}
H'(x) &\geq& q\ln{q} -1 \\
&=& \ln{q^q} -1 \\
&=& \ln{4} -1 ~~~\{\mb{since } q=2\} \\
&>& 0 \\
\end{eqnarray*}

Hence $H(x)$ is a strictly increasing function within the interval
$1 \leq x \leq i-1$. Therefore the minimum is achieved at $x=1$ and
we have
\begin{eqnarray*}
q^i - q^{i-j}-j &\geq& q^i - q^{i-1} - 1\\
&=& q^{i-1}(q-1) -1 \\
&=& q^{i-1} - 1 ~~~ \{\mb{because } q=2 \} \\
&\geq& q-1 ~~~ \{\mb{since } i \geq 2\} \\
&=& 1 
\end{eqnarray*}
\end{proof}


\begin{theorem}
\label{pal3}
Suppose $v(a_i) \geq \xqa^i-1$ for all $i \geq 1$, where 
the $a_i$ are as in equation (\ref{eq-phi}).
Then there exists a unique homomorphism $f: E \map \hG$
of $A$-module schemes over $S$, 
written $f= \sum_{i=0}^\infty b_i \tau^i$ in coordinates, with
$b_0 = 1$. Moreover,
\begin{enumerate}
	\item if $\xqa \geq 3$,  then $v(b_i) \geq i$
	and $f$ is an isomorphism of $A$-module schemes;
	\item if $\xqa =2 $, then $v(b_i) \geq i-1$.
\end{enumerate}
\end{theorem}

\begin{proof}
Let $f = \sum_{i=0}^\infty b_i \tau^i,~b_i \in R$, where $b_0 = 1$ and 
\begin{equation}
\label{recur}
b_i = \pi^{-1}(1-\pi^{\xqa^i-1})^{-1}\sum_{j=1}^i b_{i-j} 
a_j^{\xqa^{i-j}}. 
\end{equation}
Indeed, this is the only possible choice for $f$, by proposition~\ref{cor-GG}.
Conversely, it is easy to see that $f$ satisfies $\varphi(t) \circ f = f \circ \varphi(t)$,
which implies $\varphi(b) \circ f = f \circ \varphi(b)$ for all $b \in B$. 

(1) Assume $\xqa\geq 3$. Let us now show $v(b_i)\geq i$. For $i=0$, it is clear.
For $i\geq 1$, we may assume by induction that $v(b_j) \geq j$ for
all $j=1,\dots ,i-1$. 
By (\ref{recur}), we have
$v(b_i) \geq \min\{v(b_{i-j} a_j^{\xqa^{i-j}}) -1 \,|\, j=1,\dots,  i\}$. Now
\begin{eqnarray*}
v(b_{i-j}a_j^{\xqa^{i-j}})-1 &\geq& v(b_{i-j})+ v(a_j^{\xqa^{i-j}})-1 \\
&\geq& i-j + \xqa^{i-j}(\xqa^j-1)-1 \\
&=& i-j+ \xqa^i - \xqa^{i-j}-1 \\
&\geq& i, \text{ by lemma \ref{palha}.}
\end{eqnarray*}
Therefore we have $v(b_i) \geq i$.

Therefore $f$ is a restricted power series and hence defines a map between
$\pi$-formal schemes $f: E \map \hG$ which is $A$-linear.

Let us show that $f$ is an isomorphism.
By proposition \ref{pal2}, there exists a linear map $g: \hG \map 
E$ such that $f \circ g = g \circ f = \mathbbm{1}$. Then
$g$ is also $A$-linear for formal reasons: for any $a \in A$, 
we have
$f(g(\varphi(a)x))= \varphi(a)x = f(\varphi(a)g(x))$. Since $f$ is injective,
we must have $g(\varphi(a)x) = \varphi(a) g(x)$ which shows the $A$-linearity of
$g$ and we are done. 

(2) Now assume $\xqa=2$. We want to show that $v(b_i) \geq i-1$ for all $i \geq 1$. For
$i=1$, we have $b_1= \pi^{-1}(1-\pi^{q-1})^{-1}(b_0 a_1)$ and hence
$v(b_1) \geq q-1 -1 = 0$. For $i \geq 2$ and $j=1,\dots,i$,
\begin{eqnarray*}
v(b_{i-j}a_j^{q^{i-j}}) &=& v(b_{i-j}) + q^{i-j}v(a_j) \\
&\geq& (i-j-1)+ (q^i -q^{i-j}), \quad \mb{since } v(a_j) \geq q^j-1.
\end{eqnarray*}
Hence to show $v(b_i) \geq i-1$, it is enough to show that
$q^i - q^{i-j}-j \geq 0$ and that follows from lemma \ref{q2}.
\end{proof}

The remainder of this section consists of an interesting observation which will not however be used in
this paper. Letting $\Ga^{\forl}$ denote the formal completion of $\hG$ along the identity 
section $\Spf R \map \hG$. Thus we have $\Ga^{\forl} = \Spf R[[x]]$, where
$R[[x]]$ has the $(\pi, x)$-adic topology. We want to extend the $A$-action on $\Ga^{\forl}$ to
an action of $\hA$:
	\begin{equation}
	\label{extension-unram}
	\hA \map \End_{\bF_\xqa}(\Ga^{\forl}/S).
	\end{equation}
Recall that
$\End_{\bF_\xqa}(\Ga^{\forl})$ agrees with the non-commutative power-series ring $R\{\{\tau\}\}$, 
with commutation law $\tau b = b^\xqa\tau$ for $b\in R$. (See for example~\cite{drin74}, \S 2.) Therefore for any $a \in A$, we can write
\oldmarginpar{Where explained?}
	$$
	\vp(a) = \sum_j \alpha_j \tau^j
	$$
where $\alpha_j\in R$. Each $\alpha_j$ can be thought of as a function of $a\in A$. To construct
(\ref{extension-unram}) it is enough to prove that these functions are $\mfrak{p}$-adically continuous, which
also implies that such an extension to a continuous $\hA$-action is unique. This is a consequence of the 
following result.

\begin{proposition}
If $a \in \mfrak{p}^n$, then $\alpha_j \in \mfrak{p}^{n-j}R$.
\end{proposition}

\begin{proof}	
Clearly, it is true for $n=0$. Now assume it is true 
for some given $n$. Suppose $a \in \mfrak{p}^{n+1}$ and write $a=\pi b$, where $b \in \mfrak{p}^n$. Let
$\vp(b)= \sum_j \beta_j \tau^j$ and $\vp(\pi) = \sum_k \gamma_k \tau^k$. Then we have
	$$
	\sum_j \alpha_j \tau^j= \vp(a) = \vp(\pi) \vp(b) = \sum_k \gamma_k \tau^k \sum_j \beta_j \tau^j 
	= \sum_{k,j} \gamma_k \beta_j^{\xqa^k} \tau^{j+k} 
	$$
and hence $\alpha_j= \sum_{k=0}^j \gamma_k \beta_{j-k}^{\xqa^k}$. 
So to show $\alpha_j\in\fp^{n+1-j}R$, it suffices to show 
	$$
	\gamma_k\beta_{j-k}^{\xqa^k}\in\fp^{n+1-j}R, \text{ for } 0\leq k\leq j\leq n+1.
	$$
By induction we have $\beta_{j-k}\in\fp^{n-(j-k)}R$ and hence
$\gamma_k\beta_{j-k}^{\xqa^k}\in\fp^{(n-(j-k))\xqa^k}R$. Since we have $(n-(j-k))\xqa^k\geq n-j+1$ for 
$k\geq 1$, we then have $\gamma_k\beta_{j-k}^{\xqa^k}\in\fp^{n-j+1}R$. For $k=0$, because $\vp$ is a strict
module structure, we have $\gamma_0 = \pi$ and hence $\gamma_0\beta_j\in \pi\fp^{n-j}R= \fp^{1+n-j}R$.
\end{proof}

\section{Characters of $N^n$---upper bounds}

We continue to let $E$ denote the admissible $A$-module scheme over $S$ of~(\ref{eq-phi}).
Let $N^n$ denote the kernel of the projection $u:J^nE \map E$. Thus we have 
a short exact sequence of $A$-module schemes over $S$:
	$$
	0 \map N^n \map J^nE \stk{u}{\map} E \map 0
	$$
The purpose of this section is to analyze the character group of $N^n$.
In the applications of this section, $E$ will eventually be a Drinfeld module, but we do not need to 
assume this yet.

Let us fix a coordinate $x$ on $E$, and denote the corresponding Buium--Joyal coordinates on $J^nE$
by $x,x',\dots,x^{(n)}$.
From now on, let us abusively write $\phi$ for the Frobenius pull back $\phi^*$ of~(\ref{eq:phi-man}).

\begin{lemma}
\label{phi}
For all $n \geq 0$, $\phi^n(x) = \pi^n x^{(n)} + O(n-1)$, where $O(n-1)$
are elements of order less than equal to $n-1$.
\end{lemma}

\begin{proof}
For $n=0$, it is clear. For $n\geq 1$, we have by induction
	\begin{align*}
	\phi^n(x) &=\phi(\pi^{n-1}x^{(n-1)}+O(n-2)) \\
	&= \pi^{n-1}\phi(x^{(n-1)}) + O(n-1) \\
	&= \pi^{n-1}(\pi \delta(x^{(n-1)})+ (x^{(n-1)})^\xqb) + O(n-1)\\
	&= \pi^n x^{(n)} +O(n-1).
	\end{align*}
\end{proof}

\begin{theorem}
\label{kernel}
For any $n\geq 1$, let $H^n$ denote the kernel of the projection $u:J^{n}E\to 
J^{n-1}E$.  Then there is a unique $A$-linear homomorphism $\vartheta_n:H^n\to 
\hG$ of the form
$$\vartheta_n(x^{(n)})=x^{(n)}+b_1(x^{(n)})^\xqa+b_2(x^{(n)})^{\xqa^2}+\cdots,
$$ 
where $b_i\in R$.
Moreover, $\vartheta_n$ freely generates
$\HomA(H^n,\hG)$ as an $R$-module, and
\begin{enumerate}
	\item if $q \geq 3$, then $v(b_i) \geq i$ and $\vartheta_n$ is an isomorphism of $A$-module schemes;
	\item if $q=2$, then $v(b_i) \geq i-1$.
\end{enumerate}

\end{theorem}

\begin{proof}
First observe that we have 
	\begin{align*}
		\vp_E(t)\phi^n(x) &= \phi^n(\vp_E(t)) \\
			 &= \phi^n(\pi) \phi^n(x) + \phi^n(a_1) \phi^n(x)^\xqa + \dots  + \phi^n(a_r) \phi^n(x)^{\xqa^r}. 
	\end{align*}
Second, the subscheme $H^n$ is defined by setting the $x,x',\dots ,x^{(n-1)}$ coordinates to $0$. 
Combining these two observations and lemma \ref{phi}, we obtain
	$$
	\pi^n \vp_E(t)x^{(n)} =  \pi\pi^n x^{(n)} + \phi^n(a_1) (\pi^n x^{(n)})^\xqa + \dots  + 
	\phi^n(a_r) (\pi^n x^{(n)})^{\xqa^r}
	$$
and hence
	$$
	\vp_E(t)x^{(n)} = \pi x^{(n)} + \phi^n(a_1) \pi^{n(\xqa-1)} (x^{(n)})^\xqa + \dots  + 
	\phi^n(a_r) \pi^{n(\xqa^r-1)} (x^{(n)})^{\xqa^r}.
	$$
But then by theorem \ref{pal3}, there is a unique $A$-linear
homomorphism $\vartheta_n$ of the kind
desired for the respective cases of $q \geq 3$ and $q=2$.
Moreover by proposition \ref{pro:tangent}, $\HomA(H^n,\hG)$ is
freely generated by $\vartheta_n$ as an $R$-module. Finally, by proposition~\ref{pal2},
$\vartheta_n$ an isomorphism when $q \geq 3$. 
\end{proof}


Now consider the exact sequence
	$$
	0 \to H^n \to N^n \to N^{n-1} \to 0
	$$
and the corresponding long exact sequence 
\oldmarginpar{need exposition on ---homological algebra??}
	$$
	0 \map \HomA(N^{n-1},\hG) \map \HomA(N^n,\hG) \map \HomA(H^n,\hG) \map\cdots.
	$$
The image of the map $\HomA(N^n,\hG) \map \HomA(H^n,\hG)$
can be regarded as a sub-$R$-module of $R$, by theorem \ref{kernel} above.
Therefore in the $R$-module filtration
	$$
	\HomA(N^n,\hG) \supseteq \HomA(N^{n-1},\hG) \supseteq \cdots \supseteq \HomA(N^0,\hG) = 0,
	$$
each associated graded piece is canonically a submodule of $R$.

In particular, we have the following:
\begin{proposition}
\label{rk}
If $R$ is a discrete valuation ring, then $\HomA(N^n,\hG)$ is a free $R$-module of rank at most $n$.
\end{proposition}

\section{The Lateral Frobenius and characters of $N^n$}
\label{sec-lateral}
We continue to let $E$ denote the admissible $A$-module scheme over $S$ of~(\ref{eq-phi}).

Now we will construct a family of important operators which we call the {\it lateral Frobenius} operators.
That is, for all $n$, we will construct maps $\mfrak{f}:N^{n+1} \map N^n$ which are lifts of Frobenius relative
to the projections $u:N^{n+1} \map N^n$ and hence make the system $\{N^n\}_{n=0}^\infty$ into a prolongation
sequence. Do note that {\it a priori} the $A$-modules $N^n$ do not form a prolongation sequence to start with.

Let $N^\infty$ denote the inverse limit the projection maps $u:N^{n+1} \map N^n$. (Here and below, we take
inverse limits in the category of presheaves on $R$-algebras in which $\pi$ is nilpotent. They are representable by
affine formal schemes.) Then the maps $\mfrak{f}$ induce a lift of Frobenius on $N^\infty$. Similarly on
$J^\infty E = \lim_n J^nE$, the maps $\phi$ induce a lift of Frobenius. Now for all $n \geq 1$, the inclusion
$N^n \inj J^nE$ is a closed immersion and hence induces a closed immersion of schemes 
$N^\infty \inj J^\infty E$. But $\mfrak{f}$ is not obtained by restricting $\phi$ to $N^\infty$. In fact,
$\phi$ does not even preserve $N^\infty$. So $\mfrak{f}$ is an interesting operator which is distinct from
$\phi$, although it does satisfy a certain relation with $\phi$ which we will explain below.

Here we would also like to remark that the lateral Frobenius can also be constructed in the
mixed-characteristic setting of $p$-jet spaces of arbitrary schemes~\cite{bosa1},
but it is much more involved.

Let $F:W_n \map W_{n-1}$ and $V:W_{n-1} \map W_n$ denote the Frobenius 
and Verschiebung maps of~\ref{subsec-witt-operations}. 
Let us arrange them in the following diagram, although it does not commute. 
\begin{equation}
	\label{rem}
	\xymatrix{
	W_n \ar[r]^V \ar[d]_{F} & W_{n+1} \ar[d]^F\\
	W_{n-1} \ar[r]^V  & W_n \ar[d]^F\\
	& W_{n-1}
	}
\end{equation}
Rather the following is true
	\begin{equation}
	\label{FVF}
	FFV = FVF.
	\end{equation}
Indeed, the operator $FV$ is multiplication by $\pi=\theta(t)$, and $F$ is a morphism of $A$-algebras.

We can re-express this in terms of jet spaces using the natural identifications
$J^nE \simeq W_n$ and $N^n \simeq W_{n-1}$. For jet spaces, let us switch to the notation
$i:= V$ and $\phi:=F$ for the right column of~(\ref{rem}). Then we define the \emph{lateral Frobenius}
	$$
	\mfrak{f}:N^{n+1}\to N^n
	$$
simply to be the map $F:W_n\to W_{n-1}$ in left column. Thus (\ref{rem}) becomes the following:
	\begin{equation}
	\label{trapezium}
	\xymatrix{
	N^{n+1} \ar[r]^i \ar[d]_{\mfrak{f}} & J^{n+1}E \ar[d]^\phi\\
	N^n \ar[r]^i  & J^nE \ar[d]^\phi\\
	& J^{n-1}E
	}
	\end{equation}
Note again that this diagram is not commutative.
However rewriting (\ref{FVF}) in the above notation, we do have
	\begin{equation}
	\label{iphi}
	\phi^{\circ 2} \circ i = \phi \circ i \circ \mfrak{f}.
	\end{equation}


We emphasize that when we use the notation $N^n$, the $A$-module structure will always be understood to be the
one that makes $i$ an $A$-linear morphism. It should not be confused with the $A$-module structure coming by
transport of structure from the isomorphism $N^n\simeq W_{n-1}=J^{n-1}E$ of group schemes.

We also emphasize that while $i$ is a morphism of $S$-schemes, the vertical arrows
$\phi$ and $\mfrak{f}$ in the diagram above lie over the Frobenius endomorphism $\phi$ of $S$, rather
than the identity morphism.

\begin{lemma}
\label{FV}
For any torsion-free $R$-algebra $B$, the map $FV: W_n(B) \map W_n(B)$ is injective.
\end{lemma}
\begin{proof}
Since $B$ is torsion free, the ghost map $W_n(B)\to B\times\cdots\times B$ is injective,
and hence $W_n(B)$ is torsion free. The result then follows because $FV$ is multiplication by $\pi$.
\end{proof}

\begin{proposition}
\label{latfrob}
The morphism $\mfrak{f}:N^n \map N^{n-1}$ is $A$-linear.
\end{proposition}

\begin{proof}
We want to show that for any $a\in A$, the two morphisms $N^{n+1}\to N^n$ given by $x\mapsto a\mfrak{f}(x)$ and
by $x\mapsto\mfrak{f}(ax)$ are equal. Since the $N^i$ are flat over $R$, it is enough to consider $B$-points
$x$, where $B$ is a $\pi$-torsion free $R$-algebra.

Since both $\phi$ and $i$ are $A$-linear morphisms, so are $\phi i$ and $\phi^2 i$.
Therefore we have 
	$$
	\phi i(\mfrak{f}(ax)) = \phi^2 i(ax) = a \phi^2 i(x) = a \phi i(\mfrak{f}(x)) = \phi i(a \mfrak{f}(x)). 
	$$
Thus the points $\mfrak{f}(ax)$ and $a\mfrak{f}(x)$ of $N^n(B)$
become equal after the application of $\phi i$. 
Now translating from the notation of diagram (\ref{trapezium}) to that of diagram (\ref{rem}), we have
two elements of $W_{n-1}(B)$ which become equal after applying $FV$. But since $FV=\pi$ and
$B$ is torsion free, lemma \ref{FV} implies these two elements must be equal.
\end{proof}

For $0 \leq i \leq k-1$, let us abusively write $\mfrak{f}^{\circ i}$ for the following composition
	$$
	\mfrak{f}^{\circ i} :N^n \stk{\overbrace{\mfrak{f} \circ \dots  \circ 
	\mfrak{f}}^\text{$i$-times}}{\map} N^{n-i} \
	\stk{u} {\map} N^{n-k}.
	$$
Then for all $1 \leq i \leq n$, we define the \emph{canonical characters} $\Psi_i \in \HomA(N^n,\hG)$ 
(associated to our implicit coordinate $x$ on $E$) by 
	\begin{align}
		\label{def-Psi-n}
	\Psi_i = \vartheta_1 \circ \mfrak{f}^{\circ {i-1}} 
	\end{align}
where $\vartheta_1$ is as in theorem~\ref{kernel}. Clearly, the maps $\Psi_i$ are $A$-linear since each one of
the maps above is. Finally, given a character $\Psi \in \HomA(N^{n-1},\hG)$, we will write 
$\mfrak{f}^*\Psi = \Psi \circ \mfrak{f}$. Note that $\mfrak{f}^*$ is semi-linear: for $\lambda\in R$,
we have
	\begin{equation}
	\mfrak{f}^*(\lambda \Psi) = \phi(\lambda)\ \mfrak{f}^*(\Psi).
	\end{equation}

The points of $J^nE$ contained in $N^n$ are those with Witt coordinates of the form
$(0,x_1,x_2,\dots,x_n)$. We will use the abbreviated coordinates $(x_1,\dots,x_n)$ on $N^n$ instead.

\begin{lemma}
\label{psi}
For all $i=1,\dots, n$, we have 
$$\Psi_i(x_1,\dots,x_n) \equiv \left\{ \begin{array}{ll}
x_1^{\xqb^{i-1}} \bmod \pi, & \mbox{ if } q \geq 3\\
x_1^{\xqb^{i-1}} + \phi(a_1) x_1^{q{\xqb}^{i-1}}  \bmod \pi, & \mbox{ if } q=2.
\end{array}
\right.$$
where $a_1$ is the first of the structure constants of the Drinfeld 
module $E$, as in (\ref{eq-phi}).
\end{lemma}

\begin{proof}
Since $\mfrak{f}$ is identified with the Frobenius map $F:W_n\to W_{n-1}$,
it reduces modulo $\pi$ to the $\xqb$-th power of the projection map. Therefore, we have
	$$
	\Psi_i(x_1,\dots,x_n) = \vartheta_1\circ \mfrak{f}^{\circ(i-1)}(x_1,\dots,x_n) 
	\equiv \vartheta_1(x_1^{\xqb^{i-1}}) \bmod \pi.
	$$

$q \geq 3$:
By part (1) of theorem~\ref{kernel}, the map $\vartheta_1$ is congruent to the 
identity modulo $\pi$.
Therefore $\Psi_i$ is congruent to $x_1^{\xqb^{i-1}}$ modulo $\pi$.

$q =2$:
By part (2) of theorem~\ref{kernel}, we have $\vartheta_1(x_1) \equiv x_1 + b_1 x_1^q \bmod\pi$,
where by (\ref{recur}), we have 
$$
b_1=\pi^{-1}(1-\pi^{q-1})^{-1}\pi^{q-1}\phi(a_1) \equiv \phi(a_1)\bmod \pi.
$$
Therefore we have $\vartheta_1(x_1)\equiv x_1+\phi(a_1) x_1^q\bmod\pi$, and so
$\Psi_i$ is congruent to $x_1^{\xqb^{i-1}} + \phi(a_1) x_1^{q{\xqb}^{i-1}} 
\bmod \pi$.
\end{proof}

\begin{proposition}
\label{linind}
If $R$ is a discrete valuation ring, then the elements $\Psi_1,\dots,\Psi_n$ form an $R$-basis for 
$\HomA(N^n,\hG)$.
\end{proposition}
\begin{proof}
By proposition~\ref{rk}, the $R$-module $\HomA(N^n,\hG)$ is free of rank at most $n$.
So to show the elements $\Psi_1,\dots,\Psi_n$ form a basis, it is enough by Nakayama's lemma
to show they are linearly independent modulo $\pi$. 

We can view $\HomA(N^n,\hG)$ as the set of additive functions in $\sO(N^n)$. Further since $N^n$ is flat,
$\sO(N^n)$ is $\pi$-torsion free, and so any function $f\in\sO(N^n)$ is additive if $\pi f$ is. Therefore the
map $R/\pi R\otimes_R \HomA(N^n,\hG)\to R/\pi R\otimes_R\sO(N^n)$ remains injective. 

So to show they are linearly
independent in $R/\pi R\otimes_R \HomA(N^n,\hG)$, it is enough to show that
$R/\pi R\otimes_R \HomA(N^n,\hG)$ maps injectively to 
$R/\pi R\otimes_R\sO(N^n)$.
Now by lemma~\ref{psi}, we have $\Psi_i\equiv x_1^{\xqb^{i-1}}\bmod \pi$ for $q\geq 3$
(and $\Psi_i \equiv x_1^{\xqb^{i-1}} + \phi(a_1) x_1^{q\xqb^{i-1}}$ for $q=2$).
So the $\Psi_i$ map to linearly independent elements of 
$R/\pi R\otimes_R\sO(N^n)$. 
\end{proof}

\section{$\bX_\infty(E)$}
We now assume further that $R$ is a discrete valuation ring and $E$ is a Drinfeld module over $\Spf R$.
Let $r$ denote the rank of $E$. We continue to write $\vp_E(t)= a_0\tau^0+ a_1\tau^1+\dots  + a_r \tau^r$, 
where $a_0 = \pi$, $a_i \in R$ for all $i$, and $a_r\in R^*$.

In this section and the next,
we will determine the structure of $\bX_\infty(E)$. In the case of elliptic
curves, it falls in two distinct cases as to when the elliptic curve admits
a lift of Frobenius and when not. In particular, canonical lifts of ordinary elliptic curves
all fall into one case.
A similar story happens in our case when $E$ is
a Drinfeld module of rank $2$, which one might consider the closest
analogue of an elliptic curve. However, when the rank exceeds $2$, the behavior
of $\bX_\infty(E)$ offers much more interesting cases which leads us to 
introduce the concept of the {\it splitting order} $m$ of a Drinfeld module 
$E$.  The splitting order is always less than or equal to the rank of $E$. When 
the rank equals $2$, the splitting order is $1$ if and only if $E$ admits a lift of Frobenius.


We would like to point out here that our structure result for $\bX_\infty(E)$ is 
an integral version of the equal-characteristic analogue of Buium's~\cite{Bui2}. 
He shows that $\bX_\infty(E)
\otimes_R K$ is generated by a single element as a $K\{\phi^*\}$-module where
$K=R[\frac{1}{p}]$. But here we show that the module $\bX_\infty(E)$ itself is
generated by a single element as a $R\{\phi^*\}$-module. These 
methods also work in the setting of elliptic curves over $p$-adic rings, and hence this
stronger result can be achieved in that case too. (See~\cite{bosa2}.) 

The following theorem should be viewed as an analogue of the fact that
an elliptic curve has no non-zero homomorphism of $\bZ$-module schemes to $\bb{G}_\mathrm{a}$. 
In our case, we show that no Drinfeld module admits a non-zero homomorphism of $A$-module schemes
to $\hG$.

\begin{theorem}
\label{order0}
We have ${\bf X}_0(E)=\{0\}$.
\end{theorem}

\begin{proof}
Any character $f=\sum_{i\geq 0} b_i \tau^i \in \bX_0(E)$ satisfies the following chain of equalities,
where $\theta$ is as in~(\ref{eq:theta3849}):
	\begin{eqnarray*}
	\vp_{\hG}(t) \circ f &=& f \circ \vp_E(t) \\
	\theta(t)\tau^0 \circ \sum_{i\geq 0} b_i \tau^i & =& \sum_{i\geq 0} b_i \tau^i \circ 
	\sum_{j\geq 0} a_j\tau^j\\
	\sum_{i\geq 0} \theta(t) b_i \tau^i &=& \sum_{i\geq 0} \Big( \sum_{j=0}^r b_{i-j} 
	a_j^{\xqa^{i-j}}\Big) \tau^i
	\end{eqnarray*}
Comparing the coefficients of $\tau^i$ for $i > r$, and using the equality $a_0=\theta(t)$, we have
	\begin{equation}
	\label{eq-hohum}
	b_i(1-\theta(t)^{\xqa^i-1})\theta(t) = a_r^{\xqa^{i-r}}b_{i-r} + 
	a_{r-1}^{\xqa^{i-r+1}}b_{i-r+1}+ \dots  + a_1^{\xqa^{i-1}}b_{i-1}
	\end{equation}
Suppose $f$ is nonzero. There there exists an $N$ such that $b_{N-r} \ne 0$ 
and $v(b_{N-r}) < v(b_i)$ for all $ i \geq N-r+1$. Then the valuation 
of the right-hand side of equation (\ref{eq-hohum}) for $i=N$ becomes $v(a_r^{\xqa^{i-r}}b_{N-r})= v(b_{N-r})$, 
since $v(a_r)=0$. But then taking the valuation of both sides of 
(\ref{eq-hohum}), we have $$ v(b_N)= v(b_{N-r})-1 < v(b_{N-r})$$
and $N \geq N-r+1$, which is a contradiction. Therefore $f$ must be $0$. 
\end{proof}

As a consequence the short exact sequence of $A$-module schemes over $S$
	\begin{equation}
	\label{eq-humdrum}
	0 \map N^n \stk{i}{\map}  J^nE \map E \map 0,
	\end{equation}
induces an exact sequence 
	\begin{equation}
	\label{main-ex-seq}
	0 \map \bX_n(E) \stk{i^*}{\map} \HomA(N^n,\hG) \stk{\partial}{\map} \Ext_A(E,\hG),
	\end{equation}
where $\Ext_A(E,\hG)$ denotes the group of extension classes of $A$-module schemes over $R$, as
defined in Gekeler~\cite{Ge3}, section 5. He further defines an exact sequence
	\begin{equation}
	\label{seq-Hodge}
	0 \longmap \Lie(E)^* \longmap \Ext_A^\sharp(E,\hG) \longmap \Ext_A(E,\hG) \longmap 0
	\end{equation}
of $R$-modules, where $\Ext_A^\sharp(E,\hG)$ denotes the group of classes of an extension together with
a splitting of the corresponding extension of Lie algebras. Finally one defines
	\begin{equation}
		\bH_{\mathrm{dR}} (E) = \Ext_A^\sharp(E,\hG).
	\end{equation}

\begin{theorem}\label{thm:ext-rank}
	The exact sequence (\ref{seq-Hodge}) is split. The rank of $\Ext_A(E,\hG)$ is $r-1$, and
	the rank of $\Ext_A^\sharp(E,\hG)$ is $r$.
\end{theorem}
\begin{proof}
	See diagram (5.2) and corollary 3.7 in~\cite{Ge3}.
\end{proof}

The following is the equal-characteristic analogue of a result of Buium's~\cite{Bui2}, prop.\ (3.2).
\begin{theorem}
\label{X1}
Let $(E,\varphi_E)$ be a Drinfeld module of rank $r$. 
\begin{enumerate}
	\item  $\bX_r(E)$ is nonzero.
	\item We have
		$$
		\bX_1(E) \simeq \left\{\begin{array}{l} $R,$ \text{ if $E$ has a lift of Frobenius,}\\ 
		\{0\}, \text{ otherwise.} \end{array}\right.
		$$
\end{enumerate}
\end{theorem}

\begin{proof}
(1): Consider the exact sequence~(\ref{main-ex-seq}).
By proposition~\ref{linind}, the $R$-module $\HomA(N^n,\hG)$ is free of rank $n$.
But also $\Ext_A(E,\hG)$ is free of rank $r-1$, by theorem~\ref{thm:ext-rank} above.
Therefore when $n=r$, the kernel $\bX_n(E)$ is nonzero. 

(2) Now consider $\bX_1(E)$. It is contained in $\HomA(N^1,\hG)$, which is free of rank $1$,
and the quotient is contained in $\Ext_A(E,\hG)$, which is torsion free. Therefore
$\bX_1(E)$ is either $\{0\}$ or all of $\HomA(N^1,\hG)\simeq R$.

Let $\mathbbm{1}$
denote the identity map in $\HomA(\hG,\hG)$. Then its image $\partial(\mathbbm{1})$ in 
$\Ext_A(E,\hG)$ is the class of the extension (\ref{eq-humdrum}).
Therefore we have the equivalences $\bX_1(E)\simeq R \iff$ $i^*$ is an isomorphism 
$\iff \partial (\mathbbm{1}) =0  \iff$ (\ref{eq-humdrum}) is split $\iff$ $E$ has a lift of Frobenius. 
\end{proof}

Define the \emph{splitting order} of the Drinfeld module $E$ to be the integer $m$ such that
$\bX_m(E)\neq \{0\}$ and $\bX_{m-1}(E) =\{0\}$. We also say that $E$ \emph{splits at order $m$}.
By theorems~\ref{order0} and~\ref{X1} above, we have $1\leq m\leq r$ and additionally
$m=1$ if and only if $E$ has a Frobenius lift.

{\bf Computation of the character of the Carlitz module:}
Let $A = \bF_q[t]$ with $q \geq 3$. Let $E$ be the Carlitz module over $R$ satisfying 
$$\vp_E(t)(x)= \pi x + x^{\xqa}.$$
Then the operator $\vp_E(t)$ itself is a lift of Frobenius and hence, by the
universal property of $J^1E$, defines the $A$-linear splitting of the exact
sequence
$$0 \map N^1 \map J^1E \map E \map 0$$
that is, an $A$-linear morphism $\nu: J^1E \map N^1$ given in Buium--Joyal coordinates by 
$\nu(x,x') = x'-x$. Then our normalised character $\Theta_1: J^1E \map \hG$ is
given by $\Theta_1 = \vartheta_1 \circ \nu $.

Define $L_i = (\pi^{q} -\pi) \cdots (\pi^{q^i} - \pi)$. Then from theorem 
\ref{kernel}, we have $\vartheta_1:N^1 \map \hG$ given by 
\begin{equation}
\label{Carlitz1}
\vartheta_1(x') = \frac{1}{\pi}\sum_{i=0}^\infty \frac{(-1)^i}{L_i} (\pi {x'})^{q^i}.
\end{equation}
Hence we have 
\begin{equation}
\Theta_1(x,x') = \frac{1}{\pi} \sum_{i=0}^\infty \frac{(-1)^i}{L_i}
(\pi(x'-x))^{q^i} = \frac{1}{\pi}\log_C\big(\pi(x'-x)\big),
\end{equation}
where $\log_C$ denotes the Carlitz logarithm, as in~\cite{Goss}, p.\ 57.
One can check that this is the exact analogue of
Buium's character $\frac{1}{p}\log(1+p\frac{x'}{x^p})$
for $\hat{\mathbb{G}}_{\mathrm{m}}$  in the mixed-characteristic setting.

\subsection{Splitting of $J^n(E)$}
The exact sequence~(\ref{eq-humdrum}) is split by the Teichm\"uller section 
$\teich:E\to J^nE$, as defined in section~\ref{sec-witt-vectors}.
We emphasize that $\teich$ is only a morphism of $\bF_\xqa$-module schemes and is not a morphism
of $A$-module schemes. Nevertheless, it induces an isomorphism
	$$
	J^n(E) \longisomap E\times N^n
	$$
of $\bF_\xqa$-module schemes.
Therefore for any character $\Theta \in \bX_n(E)$, we can write $\Theta=g_{\Theta}\oplus\Psi_{\Theta}$ or
\begin{equation}
	\label{eq-char-split}
	\Theta(x_0,\dots ,x_n) = g_{\Theta}(x_0)+ \Psi_{\Theta}(x_1,\dots ,x_n),
\end{equation}
where $\Psi_{\Theta} = i^*\Theta \in \HomA(N^n,\hG)$ and $g_{\Theta}=\teich^*\Theta$. 
We call $g_\Theta$ the \emph{Teichm\"uller component} of $\Theta$.
Note that because $\teich$ is only $\bF_\xqa$-linear, $g_\Theta$ is also only $\bF_\xqa$-linear.
It still can, however, be expressed as an additive restricted power series. 
On the other hand, the restriction $\Psi_\Theta$ of $\Theta$ to $N^n$ does remain $A$-linear.

Now consider the morphism
	\begin{equation}
		(\phi\circ i - i\circ \mfrak{f}): N^{n+1} \longlabelmap{} J^nE,
	\end{equation}
in the notation of (\ref{trapezium}). It is an $A$-linear morphism by proposition~\ref{latfrob}.

\begin{proposition}
\label{comm}
	There exists a morphism $h$ (necessarily unique and $A$-linear) making the diagram
	$$
	\xymatrix{
		N^{n+1} \ar_-u[d]\ar^{{\phi\circ i - i\circ \mfrak{f}}}[rr]  & & J^n E \\
		N^1 \ar@{-->}_{\xg}[urr]  
	}
	$$
	commute.
In coordinates, it has the form 
	$$
	\xg(x_1)=(\pi x_1, c_1 x_1^\xqb, c_2 x_1^{\xqb^2},\dots),
	$$
for some $c_j\in R$.
\end{proposition}

\begin{proof}
By (\ref{rem})--(\ref{trapezium}), the first statement is equivalent to showing that 
for any $R$-algebra $B$, there exists a map $\xg:B \map W_n(B)$ such that
	$$
	\xymatrix{
		W_n(B) \ar[d] \ar[r]^{FV-VF} & W_n(B)\\
		B \ar@{-->}[ur]_\xg 
	}
	$$
commutes, where the vertical map is the projection onto the zeroth component.
Now for any $y \in W_{n-1}(B)$, we have
	$$
	(FV-VF)(Vy) = FVVy - VFV y  = \pi Vy - V(\pi y) =0.
	$$
So such a function $\xg$ exists.

To conclude that $\xg(x)$ is of the given form, we use a homogeneity argument. Let $(z_0,z_1,\dots)$
denote the ghost components of $(x_0,x_1,\dots)$. If interpret each $x_j$ as an indeterminate of
degree $\xqb^j$, then each $z_j$ is a homogenous polynomial in the $x_0,\dots,x_j$ of degree $\xqb^j$ and with
coefficients in $A$: $z_1=x_0^\xqb+\pi x_1$, and so on. Solving for $x_j$ in terms of $z_0,\dots,z_j$,
we see that $x_j$ is a homogenous polynomial in the $z_0,\dots,z_j$ with coefficients in $A[1/\pi]$.

Now let $(y_0,y_1,\dots)$ denote $(FV-VF)(x_0,x_1,\dots)$, where $y_j\in R[x_0,\dots,x_j]$.
Then the ghost components of $(y_0,y_1,\dots)$ are $(\pi z_0,0,0,\dots)=(\pi x_0,0,0,\dots)$.
It follows that $y_0=\pi x_0$. Further, by the above, $y_j$ is an
element of $R[x_0,\dots,x_j]$ but also a homogeneous polynomial in $\pi x_0$ of degree 
$\xqb^j$ and with coefficients in $A[1/\pi]$. Therefore it
is of the form $c_j x_0^{\xqb^j}$ for some $c_j\in R$.
\end{proof}

%
%
%

\begin{proposition}
\label{diff}
Let $\Theta$ be a character in $\bX_n(E)$. 
\begin{enumerate}
	\item We have
		$$
		i^*\phi^*\Theta = \mfrak{f}^*(i^*\Theta)+ \gamma \Psi_1,
		$$
		where $\gamma=\pi g'_\Theta(0)$ and  $g'_\Theta(x_0)$ denotes
		the usual derivative of the polynomial $g_\Theta(x_0)\in R[x_0]$
		of equation~(\ref{eq-char-split}).
	\item For $n\geq 1$, we have
		$$
		i^*(\phi^{\circ n})^*\Theta= (\mfrak{f}^{n-1})^* i^*\phi^*\Theta.
		$$
\end{enumerate}
\end{proposition}

\begin{proof}
(1): By proposition~\ref{comm}, we have 
	$$
	(\phi\circ i - i\circ \mfrak{f})(x_1,\dots,x_{n+1}) = (\pi x_1,c_1 x_1^\xqb, c_2 x_1^{\xqb^2},\dots),
	$$ 
where $c_j\in R$. Therefore we have
\begin{equation}
	\label{eq-hiho}
	\begin{aligned}
	((i^*\phi^* - \mfrak{f}^*i^*)\Theta)(x_1,\dots,x_{n+1}) &= \Theta(\pi x_1, c_1 x_1^{\xqb},\dots) \\
	&=  g_\Theta(\pi x_1) + \Psi_\Theta(c_1 x_1^\xqb,\dots).
	\end{aligned}	
\end{equation}
In particular, the character $(i^*\phi^* - \mfrak{f}^*i^*)\Theta$ depends only on $x_1$.
Therefore it is of the form $\gamma \Psi_1$, for some $\gamma\in R$. Further since by theorem~\ref{kernel} we 
have $\Psi'_1(0)=1$, the coefficient $\gamma$ is simply the linear coefficient
of $(i^*\phi^* - \mfrak{f}^*i^*)\Theta$, which by (\ref{eq-hiho}) is $\pi g'_\Theta(0)$.

(2): This is another way of expressing $\phi^{\circ n}\circ i = \phi\circ i \circ \mfrak{f}^{\circ(n-1)}$,
which follows from (\ref{iphi}) by induction.
\end{proof}

\subsection{Frobenius and the filtration by order}
We would like to fix a notational convention here. Let $u:J^nE \map J^{n'}E$
denote the canonical projection map for any $n' < n$, given in Witt coordinates by 
$u(x_0,\dots,x_n)= (x_0,\dots,x_{n'})$.

Consider the following morphism of exact sequences of $A$-modules
	$$
	\xymatrix{
	0 \ar[r] & N^n \ar@{->>}[d]_u \ar[r]^i & J^nE \ar@{->>}[d]_u \ar[r]^{u} & E \ar@{=}[d]
	\ar[r] & 0 \\
	0 \ar[r] & N^{n-1} \ar[r]^i & J^{n-1}E \ar[r]^{u} & E \ar[r] & 0. \\
	}
	$$
Since $\bX_0(E) = \{0\}$ by theorem \ref{order0}, applying 
$\HomA(-,\hG)$ to the above, we obtain the following morphism of exact sequences of $R$-modules
	$$
	\xymatrix{
	\label{longseq}
	0 \ar[r]& \bX_n(E) \ar[r]^-{i^*}& \HomA(N^n,\hG) \ar[r]^{\partial}&
	\mb{Ext}_A(E,\hG) \\
	0 \ar[r]& \bX_{n-1}(E) \ar@{^{(}->}[u]^{u^*} \ar[r]^-{i^*}& \HomA(N^{n-1},\hG) 
	\ar@{^{(}->}[u]^{u^*} \ar[r]^{\partial}& \mb{Ext}_A(E,\hG). \ar@{=}[u] \\
	}
	$$

\begin{proposition}\label{pro-filt}
	For any $n\geq 0$, the diagram
	$$
	\xymatrix{
	\bX_n(E)/\bX_{n-1}(E) \ar@{^{(}->}^{\phi^*}[r]\ar@{^{(}->}^{i^*}[d] 
		& \bX_{n+1}(E)/\bX_n(E) \ar@{^{(}->}^{i^*}[d] \\
	\HomA(N^n,\hG)/\HomA(N^{n-1},\hG) \ar^{\mfrak{f}^*}_{\sim}[r] 
		& \HomA(N^{n+1},\hG)/\HomA(N^n,\hG)
	}
	$$
	is commutative. The morphisms $i^*$ and $\phi^*$ are injective, and $\mfrak{f}^*$ is bijective.
\end{proposition}

In fact, we will show in corollary~\ref{cor:square-iso}
that all the morphisms in the diagram of proposition~\ref{pro-filt} are isomorphisms.

\begin{proof}
For $n\geq 1$, commutativity of the diagram follows from proposition~\ref{diff}; for $n=0$, it follows from
theorem~\ref{order0}. 

The maps $i^*$ are injective because the projections $J^nE \to J^{n-1}E$ and $N^n\to N^{n-1}$ have the same 
kernel, and $\mfrak{f}^*$ is an isomorphism by proposition~\ref{linind}. It follows that  $\phi^*$
is an injection. 
\end{proof}

\subsection{The character $\Theta_m$}
\label{sec:Theta-m}
Recall the exact sequence~(\ref{main-ex-seq})
	\begin{equation*}
		0\map \bX_n(E) \stk{i^*}{\map} \HomA(N^n,\hG) \stk{\partial}{\map} \Ext_A(E,\hG).
	\end{equation*}
Let $m$ denote the splitting order of $E$.
Then for all $n < m$, the map 
	$$
	\partial:\HomA(N^n,\hG) \map \Ext_A(E,\hG)
	$$ 
is injective
since $\bX_n(E)=\{0\}$. But at $n= m$, we have $\bX_m(E)\ne \{0\}$, and so there is a nonzero character 
$\Psi \in\HomA(N^m,\hG)$ in the kernel of $\partial$. Write $\Psi$ in terms of the basis of
canonical characters
$\Psi_i$ defined in~(\ref{def-Psi-n}):
	$$
	\Psi= \tilde{\lam}_m\Psi_m-\tilde{\lam}_{m-1}\Psi_{m-1} -\cdots- \tilde{\lam}_1\Psi_1,
	$$
where $\tilde{\lam}_i \in R$ for all $i=0,\dots,m-1$.
Then we necessarily have $\tilde{\lam}_m \ne 0$ since $\bX_{m-1}=\{0\}$. 
Therefore we have
	\begin{equation}
		\label{imagepsi2} 
		\partial{\Psi_m} = \lam_{m-1}\partial\Psi_{m-1}+\cdots+ 
\lam_1\partial\Psi_1 \in \Ext_A(E,\hG)_K
	\end{equation}
where $\lam_i= \tilde{\lam}_i/\tilde{\lam}_m$ for all $i= 1,\dots,m-1$. 
This implies that the character
$$\Psi_m-\lam_{1}\Psi_{m-1}-\cdots-\lam_{m-1}\Psi_1$$ is in $\ker(\partial)$ and hence by the 
main exact sequence (\ref{main-ex-seq}), 
there exists a unique $\Theta_m \in \bX_m(E)_K$ such that 
	\begin{equation}
		\label{eq-theta-def}
		i^*\Theta_m = \Psi_m -\lam_{m-1} \Psi_{m-1}-\cdots-\lam_{1}\Psi_1.
	\end{equation}
It then follows immediately that $\Theta_m$ is a $K$-linear basis for $\bX_m(E)_K$, say by
propositions~\ref{linind} and~\ref{pro-filt}.
(We will show in corollary~\ref{cor:square-iso} 
that $\Theta_m$ actually lies in the group $\bX_m(E)$ of integral characters, and is in fact an integral
basis for it.)

\begin{proposition}
\label{kerb}
Let $m$ denote the splitting order of $E$. Then
for any $j\geq 0$, the character $i^*(\phi^*)^j\Theta_m$ agrees with $\Psi_{m+j}$ modulo 
rational characters of lower order, and the elements $\Theta_m, \phi^*\Theta_m,\cdots, {\phi^{n-m}}^*\Theta_m$ 
are a basis of the $K$-vector space $\bX_n(E)_K$. 
\end{proposition}

\begin{proof}
By~\ref{pro-filt}, each character ${\phi^i}^*\Theta_m$ lies in $\bX_{m+i}(E)$ but not in $\bX_{m+i-1}(E)$.
Therefore they are linearly independent. In particular, the rank of $\bX_n(E)$ is at least $n-m+1$.

At the same time, by proposition~\ref{pro-filt}, each 
$\bX_{m+i}(E)/\bX_{m+i-1}(E)$ has rank at most $1$. Thus the rank of $\bX_n(E)$ actually equals $n-m+1$,
and so the elements in question form a $K$-basis of $\bX_n(E)_K$.
\end{proof}

Do note that this result will be improved to an integral version in theorem 
\ref{phigen-body}.

\section{Ext Groups and de Rham cohomology}
\label{sec-deRham}
We will prove theorem \ref{phigen-intro} in this section. We continue with the notation
from the previous section. In particular, $R$ is a discrete valuation ring.

We will briefly describe 
our strategy in the next few lines. Recall from (\ref{eq-theta-def}) the equality
	\begin{equation*}
		i^*\Theta_m = \Psi_m -\lam_{m-1} \Psi_{m-1}-\cdots-\lam_{1}\Psi_1
	\end{equation*}
where $\lam_j\in K$. \emph{A priori}, the elements $\lam_j$ need not belong
to $R$, but we prove in theorem \ref{intlam} that they actually do.
This implies that $i^*\Theta_m$ lies in $\HomA(N^m,\hG)$ and $\ker(\partial)$,
and hence by the exact sequence (\ref{main-ex-seq}), we have
$\Theta_m\in\bX_m(E)$---that is, the character $\Theta_m$ is integral. 
From there, it is an easy consequence that
$\bX_n(E)$ is generated by $\Theta_m,\dots, \Theta_m^{\phi^{n-m}}$ as an $R$-module.

To prove theorem \ref{intlam}, which says that all $\lambda_j$ belong to $R$, requires some preparation.
For all $n \geq 1$, we will define maps from $\HomA(N^n,\hG)$ to $\Ext^\sharp(E,\hG)$ which is also
interpreted as the de Rham cohomology from associated to the Drinfeld module $E$. These maps are obtained by
push-outs of $J^nE$ by $\Psi \in \HomA(N^n,\hG)$. 
To give an idea, do note that, for every $n \geq 1$, there
are canonical elements $E^*_\Psi \in \Ext_A(E,\hG)$ 
where the $E^*_\Psi$ is a push-out of $J^nE$ by $\Psi$ as follows
	$$
	\xymatrix{
	0 \ar[r] & N^n \ar[d]_-\Psi \ar[r]^i & J^nE \ar[d]_-{\eg_\Psi} \ar[r]^u \ar[r] & E \ar@{=}[d] \ar[r] & 0 \\
	0 \ar[r] & \hG \ar[r] & E^*_\Psi \ar[r] & E \ar[r] & 0 \\  
	}
	$$
as $E^*_\Psi \in \Ext_A(E,\hG)$. It 
leads to a very interesting theory of $\d$-modular forms over the moduli 
space of Drinfeld modules and will be studied in a subsequent paper. And similar
to previous cases, the main principles carry over to the case of elliptic 
curves or abelian schemes as well.

Now we introduce the theory of extensions of $A$-module group schemes.
Given an extension $\eta_C \in \Ext(G,T)$ and $f:T \map T'$ where $G$, $T$ and 
$T'$ are $A$-modules and $f$ is an $A$-linear map we have the following 
diagram of the push-forward extension $f_*C$.
\oldmarginpar{recall cocycles etc---homological}
	$$
	\xymatrix{
	0 \ar[r]& T \ar[r] \ar[d]_-f & C \ar[r] \ar[d] & G \ar[r] \ar@{=}[d] & 0 \\
	0 \ar[r] & T' \ar[r] & f_*C  \ar[r] & G \ar[r] & 0
	}
	$$
The class of $f_*C$ is obtained as follows---the class of $\eta_C$ is represented
by a linear (not necessarily $A$-linear) function $\eta_C:G \map T$. Then 
$\eta_{f_*C}$ is represented by the class $\eta_{f_*C}= [f \circ \eta_C] \in 
\Ext(E,T')$. In terms of the action of $t \in A$, $\vp_C(t)$ is given by
$\left(\begin{array}{ll}\vp_G(t) & 0\\ \eta_C & \vp_T(t) \end{array}\right)$ 
where $\eta_C:G \map T$. Then $\vp_{f_*C}(t)$ is given by
	\begin{align}
	\left(\begin{array}{ll}\vp_G(t) & 0\\ f(\eta_C) & \vp_{T'}(t) \end{array}\right).
	\end{align}

Now consider the exact sequence 
	\begin{equation}
		\label{jet-seq}
	0 \map N^n \stk{i}{\map} J^nE \stk{u}{\map} E \map 0	
	\end{equation}
Given a $\Psi \in\HomA(N^n,\hG)$ consider the push out
	$$
	\xymatrix{
	0 \ar[r] & N^n \ar[d]_-\Psi \ar[r]^i & J^nE \ar[d]_-{\eg_\Psi} \ar[r]^u \ar[r] 
	& E \ar@{=}[d] \ar[r] & 0 \\
	0 \ar[r] & \hG \ar[r]^i & E^*_\Psi \ar[r] & E \ar[r] & 0 \\  
	}
	$$
where $E^*_\Psi = \frac{J^nE \times \hG}{\Gamma(N^n)}$ and
$\Gamma(N^n) = \{(i(z),-\Psi(z)) | ~ z \in N^n\} \subset J^nE \times \hG$
and $\eg_\Psi(x)= [x,0] \in E^*_\Psi$.

The Teichm\"uller section $\teich:E\to J^n(E)$ is an $\bF_\xqa$-linear splitting of the sequence 
(\ref{jet-seq}). The induced retraction 
	$$
	\ret=\mathbbm{1}-\teich\circ u:J^n(E)\to N^n
	$$
is given in coordinates simply
by $\ret:(x_0,\dots,x_n)\mapsto (x_1,\dots,x_n)$.
Let us denote by $\switt$ the morphism  on Lie algebras induced by $\ret$. Thus we have
the following split exact sequence of $R$-modules
	$$
	\xymatrix{
	0 \ar[r] & \Lie N^n \ar[r]^-{Di} & \Lie J^nE \ar@/^/[l]^{{\switt}}
	\ar[r]^-{Du} & \Lie(E) \ar[r] & 0.
	}
	$$
\oldmarginpar{which coords?}
Let $s_\Psi$ denote the induced splitting of the push out extension
\oldmarginpar{rmk about tangent coords?}
	$$
	\xymatrix{
	0 \ar[r] & \Lie \hG \ar[r] & \Lie(E^*_\Psi) \ar@/^/[l]^{s_\Psi}
	\ar[r] & \Lie(E) \ar[r] & 0.
	}
	$$
It is given explicitly by $\tilde{s}_\Psi : \Lie J^nE \times \Lie\hG \map \Lie \hG$
	$$
	\tilde{s}_\Psi(x,y):= D\Psi({\switt}(x)) + y
	$$
and
	$$
	s_\Psi:\Lie(E^*_\Psi) = \frac{\Lie  J^nE \times \Lie  \hG}{\Lie \Gamma(N^n)} \map \Lie \hG.
	$$

Recall that $\Ext^\sharp(E,\hG)$ consists of an extension of $A$-module schemes together 
with a splitting $s$ of the corresponding extension of Lie algebras. (See~(\ref{seq-Hodge}) above,
or section 5 in~\cite{Ge3}.)
Therefore we have the following morphism of exact sequences
	\begin{equation}
	\label{trivext}
	\xymatrix{
		0 \ar[r] & \bX_n(E) \ar[r] \ar[d] & 
		\HomA(N^n,\hG) \ar[r] \ar[d]_{\Psi\mapsto (E_\Psi^*,s_\Psi)}& 
		\Ext(E,\hG)  \ar@{=}[d]& \\
		0 \ar[r] &\Lie(E)^* \ar[r] & \Ext^\sharp(E,\hG) \ar[r] & \Ext(E,\hG)
		 \ar[r] & 0.
	}
	\end{equation}

\begin{proposition}\label{pro:formulas}
Let $\Theta$ be a character in $\bX_n(E)$,
and put $\tilde{\Theta}=\phi^*\Theta$.
\begin{enumerate}
	\item The map $\bX_n(E)\to \Lie(E)^*$ of (\ref{trivext}) sends $\Theta$ to $-Dg_\Theta$.
	\item 
	We have $g_{\tilde{\Theta}}(x)=g_\Theta(x^\xqb)$ and
	$$\Psi_{\tilde{\Theta}}(y) = \Psi_\Theta(\ret(\phi(i(y)))) + g_\Theta(\pi y_1).$$ 
\end{enumerate}
\end{proposition}

\begin{proof}
(1):
Let us recall in explicit terms how the map is given. For the split extension $E\times\hG$,
the retractions $\Lie(E) \times\Lie \hG = \Lie(E\times\hG) \to \Lie \hG$ are in bijection with maps 
$\Lie(E)\to\Lie\hG$, a retraction $s$ corresponding to map $x\mapsto s(x,0)$. Therefore to determine
the image of $\Theta$, we need to identify $E^*_{\Psi_\Theta}$ with a split extension and then apply this map
to $s_{\Psi_\Theta}$.

A trivialization of the extension $E^*_{\Psi_\Theta}$ is given by the map
	$$
	\frac{J^nE \times \hG}{\Gamma(N^n)} = E^*_{\Psi_\Theta} \longisomap E\times\hG
	$$
defined by $[a,b]\mapsto (u(a),\Theta(a)+b)$. The inverse isomorphism $H$ is 
then given by the expression
	$$
	H(x,y) = [v(x),y-\Theta(v(x))],
	$$
and so the composition $E\to E\times\hG \to E^*_{\Psi_\Theta} \to \hG$
is simply $-\Theta\circ v=-g_\Theta$,
which induces the map $-Dg_\Theta$ on the Lie algebras.

(2): We have 
\begin{align*}
\tilde{\Theta}(x) &= \Theta(\phi(x)) \\
&= \Psi_\Theta({\ret}(\phi(t))) + g_\Theta(x_0^\xqb + \pi x_1) \\
&= \big({\Psi_\Theta}({\ret}(\phi(t)) + {g_\Theta}(\pi x_1)\big) + 
{g_\Theta}(x_0^\xqb).	
\end{align*}
In other words, we have $\tilde{\Psi} ({\ret}(x)) = {\Psi_\Theta}({\ret}(\phi(x)) + {g_\Theta}(\pi x_1)$
and $\tilde{g}(x_0) = {g_\Theta}(x_0^\xqb)$. Setting $x=i(y)$, we obtain the desired result.
\end{proof}

\begin{proposition}
\label{dual}
If $\Psi \in i^*\phi^*(\bX_n(E))$, then the class $(E^*_\Psi,s_\Psi)\in \Ext^\sharp(E,\hG)$ is zero.
\end{proposition}
\begin{proof}
Write $\Psi=i^*\phi^*\Theta$.
We know from diagram (\ref{trivext}) that $E^*_\Psi$ is a trivial extension
since ${\Psi}$ lies in $i^* \bX_{n+1}(E)$. Now by part (2) of proposition~\ref{pro:formulas},
we have ${g_{\phi^*\Theta}}(x_0)=g_\Theta(x_0^\xqb)$ and hence
$D{g_{\phi^*\Theta}}=0$. Therefore by part (1) of that proposition, the class in $\Ext^\sharp(E,\hG)$ is zero.
\end{proof}

\subsection{The $F$-crystal $\bH(E)$}
The $\phi$-linear map $\phi^*:\bX_{n-1}(E) \to \bX_n(E)$ induces a linear map
$\bX_{n-1}(E)'\to \bX_n(E)$, which we will abusively also denote $\phi^*$.
Here, for any $R$-module $M$, we write $M'$ for its base change $R\otimes_{\phi,R}M$ via $\phi:R\to R$.
We then define
	$$
	\bH_n(E) = \frac{\HomA(N^n,\hG)}{i^*\phi^*(\bX_{n-1}(E)')}.
	$$
Then 
$u:N^{n+1}
\map N^n$ induces $u^*:\HomA(N^n,\hG) \map \HomA(N^{n+1},\hG)$. And since
$u^*i^*\phi^*(\bX_n(E)) = i^*u^*\phi^*(\bX_n(E)) = i^*\phi^*u^*(\bX_n(E))
\subset i^*\phi^*(\bX_{n+1}(E))$, it also induces a map $u^*:\bH_n(E) \map \bH_{n+1}(E)$.
Define
	\begin{equation}
	\bH(E)= \varinjlim \bH_n(E),
	\end{equation}
where the limit is taken in the category of $R$-modules.

Similarly, $\mfrak{f}: N^{n+1} \map N^n$ induces $\mfrak{f}^*:\HomA(N^n,\hG) \map \HomA(N^{n+1},\hG)$, 
which descends to a $\phi$-linear morphism of $R$-modules
\oldmarginpar{move this para?}
	$$
	\mfrak{f}^*:\bH_n(E) \map\bH_{n+1}(E)
	$$
because we have $\mfrak{f}^*i^*\phi^*(\bX_{n-1}(E))= i^*\phi^*\phi^*(\bX_{n-1}(E) \subset i^*\phi^*\bX_{n}(E)$.
This then induces a $\phi$-linear endomorphism $\mfrak{f}^*:\bH(E) \map \bH(E)$. 

Finally, let $\bI_n(E)$ denote the image of $\partial$: 
	\begin{equation}
		\label{def-I-n}
		\bI_n(E) = \mathrm{im}\Big[\Hom(N^n,\hG) \longlabelmap{\partial} \Ext_A(E,\hG)\Big].
	\end{equation}
So $\Hom(N^n,\hG)/\bX_n(E)\simeq \bI_n(E)$. Then $u$ induces maps $u^*:\bI_n(E)\to \bI_{n+1}(E)$, and we put
	\begin{equation}
	\bI(E)= \varinjlim \bI_n(E),	
	\end{equation}
where again the limit is taken in the category of $R$-modules.

\begin{proposition}\label{pro:H-stable}
The morphism
	$$
	u^*: \bH_n(E)\otimes K \to \bH_{n+1}(E)\otimes K
	$$
is injective. For $n\geq m$, it is an isomorphism.
\oldmarginpar{$R$ a DVR}
\end{proposition}

\begin{proof}
Consider the following diagram of exact sequences:
	$$
	\xymatrix{
	& 0 & 0 \\
	& 
		K\langle{\phi^{\circ (n-m)}}^*\Theta\rangle' \ar[u]\ar[r]^-{i^*\phi^*} & 
		K\langle\Psi_{n+1}\rangle \ar[u] \\
	0 \ar[r] & 
		\bX_{n}(E)'_K \ar[u] \ar^-{i^*\phi^*}[r] & 
		\HomA(N^{n+1},\hG)_K \ar[u] \ar[r]& \bH_{n+1}(E)_K \ar[r] & 
		0 \\
	0 \ar[r] & 
		\bX_{n-1}(E)'_K \ar^{u^*}[u] \ar^-{i^*\phi^*}[r] & 
		\HomA(N^{n},\hG)_K \ar^{u^*}[u] \ar[r] & \bH_n(E)_K \ar[r]\ar^{u^*}[u] & 
		0 \\
	& 
		0 \ar[u] & 
		0 \ar[u] &  & 
	}
	$$
The cokernel of each of the left two maps labelled
$u^*$ is of the displayed form by propositions~\ref{linind} and~\ref{kerb}.
If $n<m$, the expression $K\langle{\phi^{\circ (n-m)}}^*\Theta\rangle$ is understood to be zero.
The map $i^*\phi^*:K\langle{\phi^{\circ (n-m)}}^*\Theta\rangle' \to K\langle\Psi_{n+1}\rangle$ is 
injective, by proposition~\ref{pro-filt}. Therefore the map $u^*:\bH_n(E)_K\to \bH_{n+1}(E)_K$
is also injective. It is an isomorphism if $n\geq m$, because $K\langle{\phi^{\circ (n-m)}}^*\Theta\rangle$ is 
$1$-dimensional and hence the map
	$$
	i^*\phi^*:K\langle{\phi^{\circ (n-m)}}^*\Theta\rangle' \to K\langle\Psi_{n+1}\rangle
	$$ 
is an isomorphism.
\end{proof}

\begin{corollary}
\label{Mn}
We have
$$
\bH_n(E) \otimes K \simeq \left\{\begin{array}{l} 
K\langle\Psi_1,\dots, \Psi_{n}\rangle,~ \mb{if } n\leq m \\
K\langle\Psi_1,\dots, \Psi_m\rangle,~ \mb{if } n\geq m
\end{array}\right.$$
\end{corollary}

Do note that we will promote this to an integral result in 
(\ref{ourdeRham2}). But before we get there, we will need some preparation.

\begin{proposition}\label{pro:I-basis}
We have
	$$
	\bI_n(E) \otimes K \simeq \left\{\begin{array}{l} 
	K\langle\Psi_1,\dots, \Psi_{n}\rangle,~ \mb{if } n\leq m-1 \\
	K\langle\Psi_1,\dots, \Psi_{m-1}\rangle,~ \mb{if } n\geq m-1
	\end{array}\right.
	$$	
\end{proposition}

\begin{proof}
The case $n\leq m-1$ is clear. So suppose $n\geq m-1$. Then
$\HomA(N^j,\hG)\otimes K$ has basis $\Psi_1,\dots,\Psi_j$,
and $\bX_n(E)\otimes K$ has basis $\Theta_m,\dots,(\phi^{n-m})^*\Theta_m$.
Since each $(\phi^j)^*\Theta_m$ equals $\Psi_{m+j}$ plus lower order terms,
$K\langle\Psi_1,\dots,\Psi_{m-1}\rangle$ is a complement to the subspace $\bX_n(E)$ of $\HomA(N^n,\hG)$.
Therefore the map $\partial$ from $K\langle\Psi_1,\dots,\Psi_{m-1}\rangle$ to the quotient $\bI_n(E)$
is an isomorphism. 
\end{proof}

Finally the morphism $\HomA(N^n,\hG)\to \Ext^\sharp(E,\hG)$ of diagram~(\ref{trivext})
vanishes on $\phi^*(\bX_{n-1}(E))$, by proposition \ref{dual}, and hence induces a morphism
of exact sequences
	\begin{equation}
	\label{diag-crys}
	\xymatrix{
	0 \ar[r] & \frac{\bX_n(E)}{\phi^*(\bX_{n-1}(E)')} \ar[r] \ar[d]_-\Upsilon & 
	\bH_n(E) \ar[r] \ar[d]^-{\Phi} & 
	\bI_n(E) \ar[r] \ar@{^{(}->}[d]& 0\\
	0 \ar[r] &\Lie(E)^* \ar[r] & \Ext^\sharp(E,\hG) \ar[r] & \Ext(E,\hG)
	 \ar[r] & 0
	}
	\end{equation}
where as in (\ref{def-I-n}), $\bI_n(E)$ denotes the image of $\partial:\Hom(N^n,\hG) \map \Ext_A(E,\hG)$.

\begin{proposition}
\label{sharp}
The map $\Phi:\bH_n(E) \otimes K \map \Ext^\sharp(E,\hG)\otimes K$ 
is injective if and only if $\gamma\neq 0$, where $\gamma\in R$ is defined as in proposition~\ref{diff}.
\end{proposition}

\begin{proof}
It is enough to show that $\Upsilon$ is injective if and only if $\gamma\neq 0$. 
By proposition~\ref{kerb}, the class of $\Theta_m$ is a $K$-linear basis for 
$\frac{\bX_n(E)}{\phi^*(\bX_{n-1}(E)')} \otimes K$,
and so it is enough to show $\Phi$ is injective if and only if $\Upsilon(\Theta_m)\neq 0$.
As in (\ref{eq-char-split}), write $\Theta_m = \Psi_{\Theta_m} + g_{\Theta_m}$. 
Then by proposition~\ref{pro:formulas}, it is enough to show
$g'_{\Theta_m}(0)\neq 0$ if and only if $\gamma\neq 0$. 
But this holds because by proposition~\ref{diff}, we have $\gamma=\pi g'_{\Theta_m}(0)$.
\oldmarginpar{repeated argument!}
\end{proof}

\begin{lemma}\label{lem:char-poly}
Consider the $\phi$-linear endomorphism $F$ of $K^m$ with matrix
	$$\left(\begin{array}{llllll}
	0 & 0 & \hdots & & 0 & \mu_{m} \\
	1 & 0 & & & 0 & \mu_{m-1} \\ 
	0 & 1 & & & 0 & \mu_{m-2} \\
	\vdots  &  &\ddots & \ddots & \vdots  & \vdots  \\
	  &   & & &   &   \\
	0 & 0 & & & 1 & \mu_1
	\end{array}\right), $$
for some given $\mu_1,\dots,\mu_m\in K$. 
If $K^m$ admits an $R$-lattice which is stable under $F$, then we have $\mu_1,\dots,\mu_m\in R$.
\end{lemma}

\begin{proof}	
We use Dieudonn\'e--Manin theory.
Without loss of generality, \oldmarginpar{why?}
we may assume that $R/\pi R$ is algebraically closed. 
Let $P$ denote the polynomial $F^m-\mu_1 F^{m-1}-\cdots-\mu_m$ in the twisted polynomial
ring $K\{F\}$. Then by (B.1.5) of~\cite{Laumon-book-vol1} (page 257), 
there exists an integer $r\geq 1$
and elements $\beta_1,\dots,\beta_m\in K(\pi^{1/r})$ such that we have
	$$
	P = (F-\beta_1)\cdots(F-\beta_m)
	$$
in the ring $K(\pi^{1/r})\{F\}$ with commutation law $F\pi^{1/r}=\pi^{1/r}F$.
(Note that the results of~\cite{Laumon-book-vol1} are stated under the assumption that the residue field of $R$
is an algebraic closure of
\oldmarginpar{check}
$\bF_p$, but they hold if it is any algebraically closed field of characteristic $p$.)
Since $R=K\cap R[\pi^{1/r}]$, it is enough to show $\mu_i\in R[\pi^{1/r}]$.
Therefore, by replacing $R[\pi^{1/r}]$ with $R$, it is enough to assume that $P$ factors
as above where in addition all $\beta_i$ lie in $K$.

Now fix $i$, and let us show $\beta_i\in R$. Assume
$\beta_i\neq 0$, the case $\beta_i=0$ being immediate. Because the (left)
$K\{F\}$-module $K^m$ has an $F$-stable integral lattice $M$, every quotient of
$K^m$ also has a $F$-stable integral lattice, namely
the image of $M$. By (B.1.9) of~\cite{Laumon-book-vol1} (page
260), for each $i$, the $K\{F\}$-module $K^m$ has a quotient (in fact, a summand) isomorphic to
$N=K\{F\}/K\{F\}(F-\pi^{v(\beta_i)})$. Therefore $N$ also has a $F$-stable integral lattice. But
this can happen only if $v(\beta_i)\geq 0$, because $F$ sends the basis element $1\in N$ to 
$\pi^{v(\beta_i)}\in N$. 
\end{proof}

\begin{theorem}
\label{intlam}
If $E$ splits at $m$, then we have $\lam_1\dots,\lambda_{m-1} \in R$, where the $\lambda_i$
are as defined in section~\ref{sec:Theta-m}.
\end{theorem}

\begin{proof}
We will prove the cases when $\gamma =0$ and $\gamma \ne 0$ separately,
where $\gamma$ is defined as in proposition~\ref{diff}.

\underline{Case $\gamma =0$}
When $\gamma =0$ we have $\mfrak{f}^* i^* = i^* \phi^*$, and hence
for all $n \geq 1$, this induces a $\phi$-linear map $\mfrak{f}^*:
\bI_{n-1}(E) \map \bI_n(E)$ as follows
	$$
	\xymatrix{
		0 \ar[r] & \bX_n(E) \ar[r]^-{i^*} & \HomA(N^n,\hG) \ar^-\partial[r] & \bI_n(E) \ar[r] & 0 \\
		0 \ar[r] & \bX_{n-1}(E) \ar[u]_\phi \ar[r]^-{i^*} & \HomA(N^{n-1},\hG) \ar^-\partial[r] 
			\ar[u]_{\mfrak{f}^*} & \bI_{n-1}(E) \ar[r] \ar[u]_{\mfrak{f}^*} & 0 \\
	}
	$$
Let $\bI(E)= \varinjlim \bI_n(E) \subseteq \Ext(E,\hG)$.
Then by proposition~\ref{pro:I-basis},
the vector space $\bI(E)_K$ has a $K$-basis $\partial\Psi_1,\dots ,\partial\Psi_{m-1}$, and
with respect to this basis, the $\phi$-linear endomorphism $\mfrak{f}^*$ has matrix
	$$
	\Gamma_0 = \left(\begin{array}{llllll}
	0 & 0 & \hdots & & 0 & \lam_{1} \\
	1 & 0 & & & 0 & \lam_{2} \\ 
	0 & 1 & & & 0 & \lam_{3} \\
	\vdots  &  &\ddots & \ddots & \vdots  & \vdots  \\
	  &   & & &   &   \\
	0 & 0 & & & 1 & \lam_{m-1}
	\end{array}\right) 
	$$
Since $\bI(E)$ is contained in $\Ext(E,\hG)$, it is a finitely generated free $R$-module and
hence an integral lattice in $\bI(E)_K$.
But then lemma~\ref{lem:char-poly} implies $\lam_1,\dots ,\lam_{m-1} \in R$.

\underline{Case $\gamma \ne 0$}
Let $\bH(E)=\varinjlim \bH_n(E)$. Let us consider the matrix $\Gamma$ of the $\phi$-linear endomorphism
$\mfrak{f}$ of $\bH(E)_K$ with respect to the $K$-basis $\Psi_1,\dots,\Psi_m$ given by corollary~\ref{Mn}. 
Then by proposition~\ref{diff} and equation (\ref{eq-theta-def}), we have
	\begin{align*}
	i^* \phi^*\Theta_m &= \mfrak{f}^*i^*\Theta_m + \gamma\Psi_1 \\
	 &= \mfrak{f}^*(\Psi_{m}-\lam_{m-1}\Psi_{m-1}-\cdots -\lam_{1}\Psi_1) + \gamma\Psi_1 \\
	 &= \mfrak{f}^*(\Psi_{m})-\phi(\lam_{m-1})\Psi_m-\cdots -
\phi(\lam_{1})\Psi_2 + \gamma\Psi_1.
	\end{align*}
Therefore we have 
	$$
	\mfrak{f}^*(\Psi_m) \equiv \phi(\lam_{m-1})\Psi_m+\cdots + \phi(\lam_{1})\Psi_2 - \gamma\Psi_1 \mod
	i^*\phi^*(\bX_{m}')
	$$
and hence
	$$ 
	\Gamma = 
	\left(\begin{array}{llllll}
	0 & 0 & \hdots & & 0 & -\gamma \\
	1 & 0 & & & 0 & \phi(\lam_{1}) \\ 
	0 & 1 & & & 0 & \phi(\lam_{2}) \\
	\vdots  &  &\ddots & \ddots & \vdots  & \vdots  \\
	0 & 0 & & & 0 & \phi(\lam_{m-2})  \\
	0 & 0 & & & 1 & \phi(\lam_{m-1})
	\end{array}\right) 
	$$
We will now apply lemma~\ref{lem:char-poly} to the operator $\mfrak{f}^*$ on $\bH(E)_K$, but to do this we need
to produce an integral lattice $M$. Consider the commutative square
	$$
	\xymatrix{
	\bH(E) \ar^-{\Phi}[r]\ar[d] & \Ext^\sharp(E,\hG) \ar^j[d] \\
	\bH(E)_K \ar^-{\Phi_K}[r]  & \Ext^\sharp(E,\hG)_K.
	}
	$$
Let $M$ denote the image of $\bH(E)$ in $\bH(E)_K$. It is clearly stable under $\mfrak{f}^*$.
But also the maps $\Phi_K$ and $j$ are injective, by proposition~\ref{sharp} and because 
$\Ext^\sharp(E,\hG)\simeq R^r$;
so $M$ agrees with the image of $\bH(E)$ in $\Ext^\sharp(E,\hG)$ and is therefore finitely generated.

We can then apply lemma~\ref{lem:char-poly} and deduce $\phi(\lam_{m-1}),\dots ,\phi(\lam_{1}) \in R$.
This implies $\lam_{m-1},\dots,\lam_{1}\in R$, since $R/\pi R$ is a field and hence 
the Frobenius map on it is injective. \oldmarginpar{ok assumptions on $R$ here?}
\end{proof} 

\begin{corollary}
\label{cor:square-iso}
\begin{enumerate}
	\item The element $\Theta_m\in\bX_m(E)_K$ lies in $\bX_m(E)$.
	\item For $n\geq m$, all the maps in the diagram 
	$$
	\xymatrix{
	\bX_n(E)/\bX_{n-1}(E) \ar^{\phi^*}[r]\ar^{i^*}[d] & \bX_{n+1}(E)/\bX_n(E) \ar^{i^*}[d] \\
	\HomA(N^n,\hG)/\HomA(N^{n-1},\hG) \ar^{\mfrak{f}^*}[r] & \HomA(N^{n+1},\hG)/\HomA(N^n,\hG)
	}
	$$
are isomorphisms.
\end{enumerate}
\end{corollary}

\begin{proof}
(1): By theorem~\ref{intlam}, the element $i^*\Theta_m$ of $\HomA(N^m,\hG)_K$ actually lies
in $\HomA(N^m,\hG)$, and therefore  by the exact sequence~(\ref{main-ex-seq}) we have $\Theta_m\in\bX_m(E)$.

(2): By proposition~\ref{pro-filt}, we know $\mfrak{f}^*$ is an isomorphism.

By proposition~\ref{pro-filt}, the maps $i^*$ are injective
for all $n\geq m$. So to show they are isomorphisms, it is enough
to show they are surjective. The $R$-linear generator $\Psi_m$ of $\HomA(N^n,\hG)/\HomA(N^{n-1},\hG)$
is the image of $\Theta_m$, which by part (1), lies in $\bX_m(E)$. Therefore $i^*$ is surjective
for $n=m$. Then because $\mfrak{f}^*$ is an isomorphism,
it follows by induction that $i^*$ is surjective for all $n\geq m$.

Finally, $\phi^*$ is an isomorphism because all the other morphisms in the diagram are.
\end{proof}

We knew before that $i^*(\phi^j)^*\Theta_m$ agrees with $\Psi_{m+j}$ plus lower order rational characters, but
the corollary above implies that these lower order characters are in fact integral.

\begin{theorem}\label{phigen-body}
	Let $E$ be a Drinfeld module that splits at $m$. 
	\begin{enumerate}
		\item For any $n\geq m$, the composition
			\begin{equation}
				\label{map-X}
				\bX_n(E) \longmap \HomA(N^n,\hG) \longmap \HomA(N^n,\hG)/\HomA(N^{m-1},\hG)				
			\end{equation}
			is an isomorphism of $R$-modules. 
		\item $\bX_n(E)$ is freely generated as an $R$-module by $\Theta_m,\dots,(\phi^*)^{n-m}\Theta_m$.
	\end{enumerate}
\end{theorem}
\begin{proof}
(i): By corollary~\ref{cor:square-iso}, the induced morphism on each graded piece is an isomorphism.
It follows that the map in question is also an isomorphism.

(ii): This follows formally from (i) and the fact, which follows from \ref{cor:square-iso}, that 
the map (\ref{map-X}) sends any $(\phi^*)^j\Theta_m$ to $\Psi_{m+j}$ plus lower order terms.
\end{proof}

\subsection{$\bH(E)$ and de Rham cohomology}
\label{subsec-ourdeRham}
Collecting the results above, we can now prove theorem~\ref{fullcrys-intro}.
Let $m$ denote the splitting order of $E$, as defined in section 8.
We have isomorphisms
	\begin{align*}
	R\langle\Psi_1,\dots,\Psi_{m-1}\rangle &= \HomA(N^{m-1},\hG) \longisomap \bI_n(E) \\ 
	R\langle\Psi_1,\dots,\Psi_m\rangle &= \HomA(N^{m},\hG) \longisomap \bH_n(E)
	\end{align*}
for $n\geq m$, and hence in the limit
	\begin{align}
	R\langle\Psi_1,\dots,\Psi_{m-1}\rangle &\longisomap \bI(E) \\ 
	R\langle\Psi_1,\dots,\Psi_m\rangle &\longisomap \bH(E)  \label{ourdeRham2}
	\end{align}
And so the $K$-linear bases of $K\otimes \bI(E)$ and $K\otimes \bH(E)$---the ones respect to which
the action of $\mfrak{f}^*$ is described by the matrices $\Gamma_0$ and $\Gamma$
in the proof of theorem~\ref{intlam}---are in fact $R$-linear bases of $\bI(E)$ and $\bH(E)$.

We also have isomorphisms for $n\geq m$
	$$
	R\langle \Theta_m\rangle = \bX_m(E) \longisomap \bX_n(E)/\phi^*(\bX_{n-1}(E)').
	$$
Combining these, we have the following map between exact sequences of $R$-modules, as in (\ref{diag-crys}):
	$$
	\xymatrix{
	0 \ar[r] & 
		\bX_m(E) \ar[r] \ar[d]_-\Upsilon & 
		\bH(E) \ar[r] \ar[d]^-{\Phi} & 
		\bI(E) \ar[r] \ar@{^{(}->}[d]& 
		0\\
	0 \ar[r] &
		\Lie(E)^* \ar[r] & 
		\Ext^\sharp(E,\hG) \ar[r] & 
		\Ext(E,\hG) \ar[r] & 
		0
	}
	$$
where $\Upsilon$ sends $\Theta_m$ to $\gamma/\pi$ (in coordinates).
It follows that $\Phi$ is injective if and only if $\gamma\neq 0$. \oldmarginpar{improve this}

\section{Computation of $\lam_1$ and $\gamma$ in the rank $2$ case}
\label{sec-computation2}

In this section, we compute $\lambda_1$ and $\gamma$ for Drinfeld modules of rank $2$, the first
nontrivial case. Recall from (\ref{eq-char-split}), proposition \ref{diff}(1), and (\ref{eq-theta-def}) that we have
	\begin{equation}
		\Theta_2 = \Psi_2(x',x'') - \lambda_1\Psi_1(x') + \pi^{-1}\gamma x + (\text{higher-degree
		terms in }x)
	\end{equation}
assuming the splitting number $m$ is $2$.
The result below shows that $\lambda_1$ and $\gamma$
depend on the higher Buium derivatives $a'_i, a''_i,\dots$ of
the modular parameters $a_i$, and not only on the modular parameters themselves. So it seems that our
$F$-crystal $\bH$ is not determined by the classical realizations, such as the crystalline realization
or the Tate module, in any straightforward manner.

\begin{theorem} 
Let $A= \bF_\xqa[v]$ with $\xqa\geq 3$, let $t \in A$ be an irreducible polynomial of degree $\rdeg$,
and let $E$ be a Drinfeld module over $R$ satisfying
	\begin{align}
	\vp_E(t)(x)= \pi x + a_1 x^\xqa + a_2 x^{\xqa^2}. \label{tact2}
	\end{align}
Then we have
	$$
	\lam_1\equiv (-1)^\rdeg w^\frac{\xqa^{\rdeg-1}(\xqa^\rdeg-1)}{\xqa-1}\big(1-a'_1 w^{\xqa^{\rdeg-1}} + a'_2 
		 	w^{\xqa^{\rdeg-1}+\xqa^{\rdeg}}\big)^{\xqa^\rdeg-1} \mod \pi,
	$$
where $w=a_1 a_2^{-1}$, and
	$$
	\gamma \mod \pi^2 \equiv  \left\{\begin{array} {ll}
\pi\lambda_{1}/a_1, & \mb{ if } f =1 \\
-\pi\lambda_1/a_2, & \mb{ if } a_1 \equiv 0 \mod \pi \mb{ and } \rdeg = 2\\
0,  & \mb{ if } a_1 \not\equiv 0 \mod \pi \mb{ or } \rdeg \geq 3 
\end{array} \right.
	$$
\end{theorem}

Observe that when $\vp_E(t)(x)$ is of the form $\pi x + a x^\xqa + x^{\xqa^2}$, which is always true 
after changing the coordinate $x$ (perhaps passing to a cover of $S$), we have the simplified forms
	\begin{align}
	\lam_1 &\equiv  (-1)^\rdeg a^\frac{\xqa^{\rdeg-1}(\xqa^\rdeg-1)}{\xqa-1}
		\big(1-a' a^{\xqa^{\rdeg-1}}\big)^{\xqa^\rdeg-1} \mod \pi,
\end{align}

\begin{align}
	\gamma \mod \pi^2 \equiv  \left\{\begin{array} {ll}
\pi\lambda_{1}/a, & \mb{ if } f=1\\
-\pi\lambda_1, & \mb{ if } a \equiv 0 \mod \pi \mb{ and } \rdeg = 2 \\
0,  & \mb{ if } a \not\equiv 0 \mod \pi \mb{ or } \rdeg \geq 3 
\end{array} \right.
\end{align}

\begin{proof}
Let $\vartheta_1:N^1 \map \hG$ be the isomorphism defined in 
theorem \ref{kernel}. Then $\vartheta_1 \equiv \tau^0 \bmod \pi$. Also $\vartheta_1$ 
induces the isomorphism $(\vartheta_1)_*:\Ext(E,N^1) \map \Ext(E,\hG)$.
In order to determine the action of $A$ on $J^1E$ and $J^2E$ we need to 
determine how $t$ acts on the coordinates $x'$ and $x''$. 
Now we note that $J^nE\simeq W_n$ can be endowed with the $\d$-coordinates 
(denoted $[z,z',z'',\dots ]$) or the Witt coordinates (denoted 
$(z_0,z_1,z_2,\dots )$) 
and they are related by the following in $J^2E$ by proposition \ref{coordinate}
	\begin{align}
	[z,z',z'']=(z,z',z''+\pi^{\xqb-2}(z')^\xqb). \label{delwitt}
	\end{align}

Taking $\pi$-derivatives of both sides of equation (\ref{tact2}) using the formula
	$$
	\delta(ax^{\xqa^j}) = a'x^{\xqb\xqa^i} + \phi(a)\pi^{\xqa^i-1}(x')^{\xqa^i},
	$$
we obtain
	\begin{equation}
		\begin{split}
	\vp(t)(x') = \pi' x^{\xqb} &+ a_1'x^{\xqa\xqb}+a_2' x^{\xqa^2\xqb}\\
		 &+ \pi x'+ \phi(a_1)\pi^{\xqa-1} (x')^\xqa + \phi(a_2)\pi^{\xqa^2-1} (x')^{\xqa^2} 					
		\end{split}
	\end{equation}
and
	\begin{equation}
	\begin{split}
	\vp(t)(x'') = \pi''x^{\xqb^2} &+ a_1''x^{\xqa\xqb^2}+a_2''x^{\xqa^2\xqb^2}\\
		&+\mb{\{terms with $x'$ and $x''$\} } \label{x''-2}
	\end{split}
	\end{equation}
Then the $A$-action $\vp_{J^1E}:A \map \End(J^1E)$ is given in Witt coordinates by the $2\times 2$ matrix
	$$
	\vp_{J^1E}(t)= 
	\left(
	\begin{array}{ll} 
		\vp_E(t) & 0 \\ 
		\eta_{J^1E} & \vp_{N^1} (t) 
	\end{array}
	\right)
	$$
where $\eta_{J^1E} = \pi' x^{\xqb} + a_1'x^{\xqa\xqb}+a_2' x^{\xqa^2\xqb}$. 
And by (\ref{x''-2}) and (\ref{delwitt}), the $A$-action 
$A\map \End(J^2E)$ is given by the $(1+2)\times(1+2)$ block matrix
	\begin{align*}
	\vp_{J^2E}(t)=
	\left(\begin{array}{ll} \vp_E(t) & 0 \\ \eta_{J^2E} & \vp_{N^2}(t)
	\end{array}\right)
	\end{align*}
where (using~\ref{delwitt}) $\eta_{J^2E}$ is the column vector
	$$
	\eta_{J^2E} =
	\left(
	\begin{array}{c} 
		\pi' x^{\xqb} + a_1'x^{\xqa\xqb}+a_2' x^{\xqa^2\xqb} \\ 
		\Delta(\pi)x^{\xqb^2} + \Delta(a_1) x^{\xqa\xqb^2} + \Delta(a_2) x^{\xqa^2\xqb^2}
	\end{array}
	\right)
	$$
and where $\Delta(y)= y''+ \pi^{\xqb-2}(y')^\xqb$.

Now we will consider two cases---

(1): Consider $\eta_{{\Psi_1}_*(J^1E)} \in \Ext(E,\hG)$ which is the image of 
$\Psi_1$ under the connecting morphism $\HomA(\hG,\hG) \stk{\partial}
{\map} \Ext(E,\hG)$ and $\Psi_1=\vartheta_1:N^1 \map \hG$ is the isomorphism defined in 
theorem \ref{kernel} and satisfies $\Psi_1=\tau^0 \mb{ mod } \pi$.
	$$
	\xymatrix{
	0 \ar[r]& N^1 \ar[r] \ar[d]_-{\Psi_1} & J^1E \ar[r] \ar[d] & E \ar[r] 
	\ar@{=}[d] & 0 \\
	0 \ar[r] & \hG \ar[r] & f_*(J^1E)  \ar[r] & E \ar[r] & 0 
	}
	$$
where $\eta_{J^1E} = [\pi' x^{\xqb} + a_1'x^{\xqa\xqb}+a_2' x^{\xqa^2\xqb}]\in \Ext(E,N^1)$
Hence 
	\begin{align*}
	\eta_{{\Psi_1}_*(J^1E)} =& [\pi' x^{\xqb} + a_1'x^{\xqa\xqb}+a_2' x^{\xqa^2\xqb}] 
	\in \Ext(E,\hG) \nonumber \\
	\partial(\Psi_1) \equiv& [x^{\xqb} + a_1'x^{\xqa\xqb}+a_2' x^{\xqa^2\xqb}] \mb{ mod } \pi. 
	\end{align*}

(2): Now consider $\eta_{{\Psi_2}_*(J^2E)} \in \Ext(E,\hG)$ obtained as
	$$\xymatrix{
	0 \ar[r]& N^2 \ar[r] \ar[d]_-{\Psi_2} & J^2E \ar[r] \ar[d] & E \ar[r] 
	\ar@{=}[d] & 0 \\
	0 \ar[r] & \hG \ar[r] & f_*(J^2E)  \ar[r] & E \ar[r] & 0 
	}$$
Now we have
	$$
	\eta_{J^2E} =
	\left[\left(
	\begin{array}{c} 
		\pi' x^{\xqb} + a_1'x^{\xqa\xqb}+a_2' x^{\xqa^2\xqb} \\ 
		\Delta(\pi)x^{\xqb^2} + \Delta(a_1) x^{\xqa\xqb^2} + \Delta(a_2) x^{\xqa^2\xqb^2}
	\end{array}
	\right)\right]
	\in \Ext(E,N^2)
	$$
Let ${\mathcal{F}}(y)= (y')^\xqb+ \pi \Delta(y)$. Then
applying $\Psi_2 = \vartheta_1 \circ \mfrak{f}$ and $\mfrak{f}(z_1,z_2)= z_1^\xqb+ \pi z_2$, we have
	\begin{align*}
	\partial(\Psi_2)= \eta_{{\Psi_2}_*(J^2E)} 
		&= [\vartheta_1({\mathcal{F}}(\pi) x^{\xqb^2}+ {\mathcal{F}}(a_1)x^{\xqa\xqb^2}+ {\mathcal{F}}(a_2)x^{\xqa^2\xqb^2})] 
		\in \Ext(E,\hG) \nonumber \\
	\partial(\Psi_2) &\equiv [{\mathcal{F}}(\pi) x^{\xqb^2}+ {\mathcal{F}}(a_1)x^{\xqa\xqb^2}+ {\mathcal{F}}(a_2)x^{\xqa^2\xqb^2}] \bmod \pi\\
		&\equiv [(\pi')^\xqb x^{\xqb^2}+ (a'_1)^\xqb x^{\xqa\xqb^2}+ (a'_2)^\xqb x^{\xqa^2\xqb^2}] \bmod \pi \\
		&\equiv [x^{\xqb^2}+ (a'_1)^\xqb x^{\xqa\xqb^2}+ (a'_2)^\xqb x^{\xqa^2\xqb^2}] \bmod \pi. 
	\end{align*}

Recall (\cite{Ge3}, section 5) that the map $R\{\tau\}^{\hat{}}\to\Ext(E,\hG)$ given by $\eta\mapsto [\eta]$ 
is surjective and the kernel consists of the inner derivations, which is to say all $\eta$ of the form 
	$$
	\pi \alpha- \alpha \circ \vp_E(t),
	$$
for some $\alpha\in R\{\tau\}^{\hat{}}$. 
Let us now work out these relations explicitly for $\alpha=\tau^0,\tau^1,\tau^2$.
If $\alpha = \tau^j$, with $j\geq 0$, we get the relation
	\begin{align*}
	\pi \tau^j &= \tau^j(\pi \tau^0+a_1\tau^1 + a_2\tau^2) \\
	\tau^{j+2} &= a_2^{-\xqa^j}[(\pi-\pi^{\xqa^j})\tau^j - a_1^{\xqa^j}\tau^{j+1} ] \\
	\tau^{j+2} &\equiv -(a_1a_2^{-1})^{\xqa^j} \tau^{j+1} \mb{ mod } \pi 
	\end{align*}
and hence we have by induction the relations
	\begin{equation}
	\tau^{i+1} \equiv (-1)^{i}w^\frac{\xqa^{i}-1}{\xqa-1} \tau^1 \bmod \pi \label{tauj}
	\end{equation}
where $w=a_1 a_2^{-1}$, for all $i\geq 0$.

Therefore writing $\xqb=\xqa^\rdeg$, we have	
	\begin{align*}
	\partial(\Psi_1)  &\equiv 	x^{\xqb} + a_1'x^{\xqa\xqb}+a_2' x^{\xqa^2\xqb} \\
	&\equiv 	x^{\xqa^\rdeg} + a_1'x^{\xqa^{\rdeg+1}}+a_2' x^{\xqa^{\rdeg+2}} \\
	&\equiv \tau^\rdeg + a'_1\tau^{\rdeg+1} + a'_2\tau^{\rdeg+2} \\
	&\equiv 
	(-1)^{\rdeg+1}w^{1+\cdots+\xqa^{\rdeg-2}}(1-a'_1w^{\xqa^{\rdeg-1}}+a'_2w^{\xqa^{\rdeg-1}+\xqa^{\rdeg}})\tau^1
	\end{align*}
and
	\begin{align*}
	\partial(\Psi_2) & \equiv	x^{\xqb^2}+ (a'_1)^\xqb x^{\xqa\xqb^2}+ (a'_2)^\xqb x^{\xqa^2\xqb^2}\\
		& \equiv \tau^{2\rdeg} + (a'_1)^{\xqa^\rdeg}\tau^{2\rdeg+1} + (a'_2)^{\xqa^\rdeg} \tau^{2\rdeg+2} \\
		& \equiv (-1)^{2\rdeg+1}w^{1+\cdots+\xqa^{2\rdeg-2}}
			\big(1-(a'_1)^{\xqa^\rdeg} w^{\xqa^{2\rdeg-1}} 
			+ (a'_2)^{\xqa^\rdeg} w^{\xqa^{2\rdeg-1}+\xqa^{2\rdeg}}\big) \tau^1 \\
		& \equiv (-1)^{2\rdeg+1}w^{1+\cdots+\xqa^{2\rdeg-2}}
			\big(1-a'_1 w^{\xqa^{\rdeg-1}} + a'_2 w^{\xqa^{\rdeg-1}+\xqa^{\rdeg}}\big)^{\xqa^\rdeg} \tau^1.
	\end{align*}
and hence
	\begin{align*}
	\lam_1 &= \frac{\partial(\Psi_2)}{\partial(\Psi_1)} \equiv 
	 (-1)^\rdeg w^{\xqa^{\rdeg-1}+\cdots+\xqa^{2\rdeg-2}} \big(1-a'_1 w^{\xqa^{\rdeg-1}} + a'_2 
	 	w^{\xqa^{\rdeg-1}+\xqa^{\rdeg}}\big)^{\xqa^\rdeg-1} \mb{ mod } \pi \\
	&\equiv  (-1)^\rdeg w^{\xqa^{\rdeg-1}(1+\cdots+\xqa^{\rdeg-1})} \big(1-a'_1 w^{\xqa^{\rdeg-1}} + a'_2 
		 	w^{\xqa^{\rdeg-1}+\xqa^{\rdeg}}\big)^{\xqa^\rdeg-1} \mb{ mod } \pi	\\
	&\equiv  (-1)^\rdeg w^\frac{\xqa^{\rdeg-1}(\xqa^\rdeg-1)}{\xqa-1}\big(1-a'_1 w^{\xqa^{\rdeg-1}} + a'_2 
		 	w^{\xqa^{\rdeg-1}+\xqa^{\rdeg}}\big)^{\xqa^\rdeg-1} \mb{ mod } \pi	
	\end{align*}

Now we determine $\gamma$. Write $g=g_{\Theta_2}=\sum_i \alpha_i \tau^i$.
Then from proposition~\ref{diff}, we know $\gamma= \pi \alpha_0$. Now we will compute $\alpha_0$. Let 
$(z_0,z_1,z_2):= \vp_{J^2E}(t)(x,0,0)$. Then
	\begin{align}
	\Theta_2(\vp_{J^2E}(t)(x,0,0)) &= \Psi_2(z_1,z_2) - \lam_1 \Psi_1(z_1) + g(z_0)
	\nonumber \\
	&= \vartheta_1(z_1^\xqb+ \pi z_2) - \lam_1 \vartheta_1(z_1) + g(z_0) \nonumber \\
	&\equiv z_1^\xqb -\lam_1 z_1 + g(z_0) \mb{ mod } \pi \nonumber
	\end{align}
where $z_0 = \pi x + a_1x^\xqa + a_2x^{\xqa^2}$ and $z_1=\pi' x^{\xqb} + a_1'x^{\xqa\xqb}+a_2' x^{\xqa^2\xqb}$. 
On the other hand from the $A$-linearity of $\Theta_2$ we have
	\begin{equation*}
	\Theta_2(\vp_{J^2E}(t)(x,0,0)) = \vp_{\hG}(t)\Theta_2(x,0,0)= \pi \Theta_2(x,0,0) \equiv 0 \mb{ mod } \pi
	\end{equation*}
and hence $z_1^\xqb -\lam_1 z_1 + g(z_0) \equiv 0 \bmod \pi$.
Substituting $z_0$ and $z_1$ in, we obtain
\begin{align*}
	0 &\equiv (\pi' x^{\xqb} + a_1'x^{\xqa\xqb}+a_2' x^{\xqa^2\xqb})^\xqb -
		\lam_1( \pi' x^{\xqb} + a_1'x^{\xqa\xqb}+a_2' x^{\xqa^2\xqb}) + g(\pi x + a_1x^\xqa + a_2x^{\xqa^2}) \\
	&\equiv (x^{\xqb} + a_1'x^{\xqa\xqb}+a_2' x^{\xqa^2\xqb})^\xqb -\lam_1 (x^{\xqb} + a_1'x^{\xqa\xqb}+a_2' x^{\xqa^2\xqb}) + g(a_1x^\xqa + a_2x^{\xqa^2}) 	
\end{align*}
Now substitute $g(x)=\sum_{j\geq 0} \alpha_j x^{\xqa^j}$ into this and consider
the coefficient of $x^\xqa$. If $\xqb=\xqa$, we obtain $\lambda_{1}\equiv\alpha_0 a_1$ and hence
	$$
	\gamma = \pi \alpha_0 \equiv \pi\lambda_{1}/a_1 \bmod \pi^2.
	$$
If $\xqb\neq \xqa$, we obtain $\alpha_0a_1\equiv 0$ and hence $\gamma\equiv 0 
\bmod \pi^2$ if $a_1 \not\equiv 0 \mod \pi$. If $a_1 \equiv 0 \mod \pi$,
 we consider the 
coefficient of $x^{\xqa^2}$ which is $\alpha_0 a_2+ \lambda_1$ if $\rdeg =2$
and $\alpha_0a_2$ otherwise. In the case when $\rdeg =2$ we have $\alpha_0 
\equiv  \lambda_1/a_2 \mod \pi$ since $a_2$ is invertible and hence 
$\gamma \equiv -\pi \lambda_1/a_2 \mod \pi^2$. When $f \geq 3$ we have 
$\alpha_0 \equiv 0 \mod \pi$ and hence the result follows.
\end{proof}

\footnotesize{

}

\end{document}